\documentclass[12pt,a4paper]{article}
\usepackage{amsmath,amssymb,amsthm,graphicx,color}
\usepackage[left=1in,right=1in,top=1in,bottom=1in]{geometry}
\usepackage{cite}
\usepackage[british]{babel}
  
\newtheorem{theo}{Theorem}[section]
\newtheorem{lm}{Lemma}[section]
\newtheorem{df}{Definition}[section]

\newtheorem{rmk}{Remark}[section]
\newtheorem{proposition}{Proposition}[section]

\allowdisplaybreaks

\numberwithin{equation}{section}

\def\d{\delta}
\def\g{\gamma}
\def\eps{\varepsilon}

\def\R{{\mathbb R}}
\def\Z{{\mathbb Z}}
\def\N{{\mathbb N}}
\def\SS{{\mathcal S}}
\def\Sp{{\mathbb{S}}}
\def\Exp{\mathbb{E}}
\def\Pr{\mathbb{P}}
\def\1{{\mathbf 1}}

\def\essinf{\mathop{{\rm ess~inf}}}
\def\esssup{\mathop{{\rm ess~sup}}}
\def\eps{\varepsilon}

\newcommand{\as}{{ \ \textrm{a.s.}}}
\def\be{{\mathbf e}}
\def\bb{{\mathbf b}}
\def\bx{{\mathbf x}}
\def\by{{\mathbf y}}

\def\bu{{\mathbf u}}
\def\0{{\mathbf 0}}

\newcommand{\ud}{{\mathrm d}}
\newcommand{\F}{{\mathcal F}}

\newcommand{\toas}{\stackrel{{\rm a.s.}}{\longrightarrow}}

\title{Random walk with barycentric self-interaction}

\author{Francis Comets\footnote{Universit\'e Paris 7 (Diderot), Case courrier 7012,
2 Place Jussieu,
75251 Paris Cedex 05, France. E-mail: \texttt{comets@math.jussieu.fr}.}
\and Mikhail V. Menshikov\footnote{Department of Mathematical Sciences, University of Durham, South Road, Durham DH1 3LE, UK. E-mail:
\texttt{mikhail.menshikov@durham.ac.uk}.}
\and Stanislav Volkov\footnote{Department of Mathematics, University of Bristol, University Walk, Bristol BS8 1TW, UK. E-mail:
\texttt{s.volkov@bristol.ac.uk}.}
\and Andrew R. Wade\footnote{Department of Mathematics and Statistics, University of Strathclyde, 26 Richmond Street, Glasgow G1 1XH, UK. E-mail: \texttt{andrew.wade@strath.ac.uk}.}\ \footnote{Corresponding author. Tel:  +44 (0)141 548 3663. Fax: +44 (0)141 548 3345.}}
\begin{document}

\maketitle

\begin{abstract}
We study the asymptotic behaviour
of a $d$-dimensional self-interacting
random walk $(X_n)_{n \in \N}$ ($\N := \{ 1,2,3,\ldots \}$)
which is
repelled or attracted by the centre of mass $G_n = n^{-1} \sum_{i=1}^n X_i$
 of
its previous trajectory. The walk's trajectory  $(X_1,\ldots,X_n)$
models a random polymer chain
in either poor or good solvent.
In addition to some natural regularity conditions, we assume that the walk has one-step mean
drift
\[ \Exp [X_{n+1} - X_n \mid X_n - G_n = \bx] \approx \rho \|\bx\|^{-\beta} \hat \bx \]
for $\rho \in \R$ and $\beta \geq 0$. When $\beta <1$ and $\rho>0$,
we show that $X_n$ is transient with a limiting (random) direction
and satisfies a super-diffusive law of large numbers:
$n^{-1/(1+\beta)} X_n$ converges almost surely to some  random vector.
When $\beta \in (0,1)$ there is sub-ballistic rate of escape.
For $\beta \geq 0$, $\rho \in \R$ we give almost-sure bounds on the norms $\|X_n\|$, which in the context of
the  polymer model reveal   extended  and collapsed phases.

Analysis  of the random walk, and in particular of
$X_n - G_n$,
 leads to the study of real-valued
 time-inhomogeneous
non-Markov
processes $(Z_n)_{n \in \N}$
 on $[0,\infty)$
with  mean drifts of the form
\begin{equation}
\label{star}
 \Exp [ Z_{n+1} - Z_n \mid Z_n = x ] \approx \rho x^{-\beta} - \frac{x}{n},  \end{equation}
where $\beta \geq 0$ and $\rho \in \R$. The study of such processes is
a time-dependent
variation on a classical problem of Lamperti; moreover, they arise naturally in the
context of the distance of simple random walk on $\Z^d$ from its centre of mass,
for which we also give an apparently new result.
We give a recurrence classification and asymptotic theory for processes $Z_n$
satisfying (\ref{star}), which enables us to deduce
the complete recurrence classification (for any $\beta \geq 0$)
of $X_n - G_n$ for our self-interacting walk.
\end{abstract}

\smallskip
\noindent
{\em Keywords:} Self-interacting random walk;
self-avoiding walk; random walk avoiding its convex hull; random polymer;
centre of mass; simple random walk; random walk average; limiting direction; law of large numbers. \/

\noindent
{\em AMS 2010 Subject Classifications:} 60J05 (Primary) 60K40; 60F15; 82C26 (Secondary)

  \section{Introduction}
  \label{sec:intro}

  We study a  self-interacting
  random walk.
  Self-interacting random processes,
 in which the stochastic behaviour  depends on the entire previous history
  of the process,
  present many challenges for mathematical
  analysis (see e.g.~\cite{blr,pemantle} and references therein)
  and are often motivated by real
  applications.

  Although not a random process in the same sense,  the {\em self-avoiding walk} is
  a  prototypical example
of a self-interacting random walk that gives rise to important and difficult problems.
  Random self-avoiding
   walks were introduced to
  model the configuration of
  polymer molecules in solution.
  The sites visited by the walk represent
  the locations of the polymer's constituent monomers;
  successive monomers are viewed as connected
  by chemical bonds.
  The classical self-avoiding walk
  (SAW) model
  takes uniform measure on $n$-step self-avoiding
  paths in $\Z^d$. In the important
  cases of $d \in \{2,3\}$, there are still
 major open problems for such walks: see for example
   \cite{lsw,ms} and \cite[Chapter 7]{hughes}, or \cite[Chapter 7]{rg}
   for a mathematical physics perspective.

   The {\em loop-erased random walk} (LERW), obtained by erasing chronologically the
loops of a random walk, was introduced in \cite{lawler} to study SAW, but it was
soon realized that the two processes belong to different universality classes.
For its independent interest, including applications to combinatorics and
quantum
field physics, LERW has received considerable attention
and now there is a more precise picture of its behaviour, which shows fine dependence on the spatial dimension.
In the planar case, the mean number of steps for LERW stopped at distance $n$ is of order
$n^{5/4}$ \cite{kenyon}, and the scaling limit is conformally invariant,
described by the radial Schramm--Loewner
evolution with parameter 2 \cite{lsw2}.

   A different perspective on polymer models concerns directed
polymers, where the self-interaction is reduced to a trivial form
but interesting phenomena arise from the interaction with the
medium: see \cite{giac,holl} for recent surveys for localization on interfaces (pinning,
wetting) possibly with time-inhomogeneities (e.g.~copolymers), and
 \cite{CSY04} for interactions with a time-space inhomogeneous medium
leading to localization in the bulk.

In the standard framework,  SAW cannot be interpreted as a
dynamic (or progressive) stochastic  process. There have been many attempts
to formulate genuine stochastic processes with similar
behaviour   to that of, or at least conjectured for, SAW.
A recent
model  is the random walk on $\R^2$   which
at each step avoids the convex hull of   its preceding values
\cite{abv,zern}.   Unlike the conjectured behaviour   of SAW, this
model is ballistic (see  \cite{abv,zern}), i.e.,   it has a
positive speed.
   The discrete version on $\Z^2$,
the dynamic prudent walk, has been studied in \cite{BFV09}: it is ballistic with speed
$3/7$ (in the $L^1$ norm), but, in contrast to the (conjecture for the)   continuous
model, it does not have a fixed direction (see \cite{BFV09}).
Ballisticity is known for other types
of self-interacting random walks: see \cite{chayes,iv}.

  In this paper we consider a self-interacting random walk model
  that is a tractable alternative to  SAW, and is distinguished
   from the models \cite{abv,BFV09,zern,chayes,iv} by exhibiting
   a range of possible scaling behaviour, including
   sub-ballisticity (i.e., zero speed) and super-diffusivity.
      Our model is tunable, with parameters that in principle can be estimated from real
data, and it can be used to represent polymers
     in the extended phase (for good solvent)
   or collapsed phase (poor solvent).
   The self-interaction in the model at time $n$ is mediated through the {\em barycentre}
   or {\em centre of mass} of the past trajectory  until time $n$.
   Specifically,
   our random walk will at each step have a mean drift
   (typically asymptotically zero in magnitude)
   pointing away from or towards
   the average of all previous positions. We now informally describe the probabilistic model;
    we give
  a brief description of the motivation and interpretation
  arising from polymer physics in Section \ref{poly}.

Let $d \in \N := \{1,2,3,\ldots\}$.
 Our random walk will be a discrete-time stochastic
  process $X = (X_n)_{n \in \N}$ on $\R^d$.
For $n \in \N$,
  set
  \begin{equation}
\label{com}
   G_n := \frac{1}{n} \sum_{i=1}^n X_i,
   \end{equation}
  the centre of mass (average)
  of $\{X_1,\ldots,X_n\}$.
  In addition to some regularity conditions on $X$
  that we describe later, our main assumption will be that
  the one-step mean drift of the walk after $n$ steps
   is of order $\| X_n - G_n \|^{-\beta}$
  in the direction  $\pm(X_n - G_n)$, where $\beta \geq 0$ is a fixed parameter;
  here and subsequently $\| \cdot \|$ denotes the Euclidean norm on $\R^d$.
Loosely speaking for the moment, we will
  suppose
   that for some $\rho \in \R$ and $\beta \geq 0$,
 \begin{equation}
 \label{drift0}
    \Exp [ X_{n+1} - X_n \mid X_n - G_n = \bx ] \approx \rho \| \bx \|^{-\beta} \hat \bx ,
    \end{equation}
 for any $n \in \N$ and $\bx \in \R^d \setminus \{ \0\}$,
 where $\hat \bx := \bx/\| \bx \|$ denotes a unit vector in the $\bx$-direction
 and $\0$ is the origin in $\R^d$. We attach no precise meaning to `$\approx$' in
 (\ref{drift0}) (or elsewhere); it indicates that we are ignoring some terms and
 also that we have not yet formally defined all the terms present.
  We describe the model
 formally and in detail in Section \ref{sec:model} below.

 The natural case of our model to compare to the walk that avoids its  convex hull \cite{abv,zern}
 has $\beta =0$ and $\rho >0$, when our walk has positive
 drift away from its current centre of mass. In our $\beta =0$, $\rho >0$ setting we   show that
 the walk has  an asymptotic speed and an asymptotic direction,
 properties which are conjectured but not yet proved for the walk avoiding its convex hull
 \cite{abv,zern}. Our results however cover much more than this special case. For example, the case of our model that we might expect to be in some sense
 comparable to SAW in $d=2$ has $\beta = 1/3$, $\rho >0$: see
 the
 discussion in Section \ref{poly} below.

 To give a flavour of our more general results,   described in more detail in  Section \ref{results} below,
 we now informally describe our results in
 the case where (\ref{drift0}) holds with $\rho>0$ and $\beta \in [0,1)$. Under
 suitable regularity conditions, we show that $X$
 is transient, i.e.~$\| X_n \| \to \infty$ a.s., and moreover
 we prove a strong law of large numbers that precisely
 quantifies this transience: $n^{-1/(1+\beta)} \| X_n \|$ is asymptotically
 constant, almost surely. In addition, we show that $X_n$ has a limiting direction,
 that is, $X_n / \| X_n \|$ converges a.s.~to some (random) unit vector.
 Thus we have, in this case, a rather complete picture of the asymptotic
 behaviour of $X_n$. For other regions of the $(\rho,\beta)$ parameter
 space we have other results, although we also leave some interesting
 open problems.

  The self-interaction in the model
  is introduced via the presence of $G_n$
  in (\ref{drift0}). If the condition $\{X_n - G_n = \bx\}$
  in (\ref{drift0}) is replaced by $\{ X_n = \bx\}$ then there
  is no self-interaction in the drift, which instead points away from a fixed origin.
  Such non-homogeneous `centrally biased' walks were
  studied by Lamperti in \cite[Section 4]{lamp1} and \cite[Section 5]{lamp3}; for
  more recent work see e.g.~\cite{flp,mmw,superlamp}.
Considering the
  process of norms $Z_n = \| X_n \|$ leads
  to a  process on $[0,\infty)$ with mean drift
   \begin{equation}
 \label{drift0a}
    \Exp [ Z_{n+1} - Z_n \mid Z_n  = x ] \approx \rho' x^{-\beta},
    \end{equation}
    ignoring higher-order terms.
  Such `asymptotically zero-drift' processes are of independent interest;
  the  asymptotic analysis of such (not necessarily Markov)
  processes
  is sometimes known as {\em Lamperti's problem} following pioneering work
of Lamperti \cite{lamp1,lamp2,lamp3}. From the point of view of the recurrence
classification of processes satisfying (\ref{drift0a}), the
case $\beta=1$ turns out to be critical, in which case the value of $\rho' \in \R$
is crucial: we give a brief summary of the relevant background in Section \ref{sec:lampsrw} below.

  We shall see below that considering the
  process  $Z_n = \| X_n -G_n \|$ with $X_n$ satisfying
   (\ref{drift0})
  leads to a more complicated form of (\ref{drift0a}).
 Loosely speaking, we will obtain
   \begin{equation}
 \label{drift0b}
    \Exp [ Z_{n+1} - Z_n \mid Z_n  = x ] \approx \rho' x^{-\beta} - \frac{x}{n}.
    \end{equation}
  We note that the two terms
    on the right-hand side of (\ref{drift0b}) are typically
    of the same order, as can be predicted by solving the corresponding differential equation,
     and so both contribute to the asymptotic behaviour.

 Comparing (\ref{drift0b}) with (\ref{drift0a}), we see that
  the drift is now {\em time}- as well as space-dependent.
 (A different variation on (\ref{drift0a})
  with this property was studied
  in \cite{mv2}, where processes with drift $\rho x^\alpha n^{-\beta}$ were considered.)
  Thus (\ref{drift0b}) is an interesting starting point for analysis
  in its own right. Additional motivation for (\ref{drift0b})
  arises naturally from simple random walk (SRW) and its centre of mass: if $Z_n = \| X_n - G_n \|$ where $X_n$ is a symmetric
 SRW on $\Z^d$ and $G_n$ its centre-of-mass
  as defined by (\ref{com}), $Z_n$ satisfies (\ref{drift0b}) with $\beta=1$
  and $\rho' = \rho' (d)$; see Section \ref{sec:com} below.

  Let us step back from the general setting for a moment
  to state one consequence of our results,
  which is a (seemingly new) observation  on SRW:

    \begin{theo}
    \label{srwthm}
    Let $d \in \N$.
  Suppose that $(X_n)_{n \in \N}$ is a symmetric SRW on $\Z^d$, and
  $(G_n)_{n \in \N}$ is its centre-of-mass process
  as defined by (\ref{com}).
  Then
  \begin{itemize}
  \item[(a)] $\liminf_{n \to \infty} \| X_n - G_n \| < \infty$ a.s.~for $d \in \{1,2\}$;
  \item[(b)] $\lim_{n \to \infty} \| X_n - G_n \| = \infty$ a.s.~for $d \geq 3$.
  \end{itemize}
    \end{theo}

  P\'olya's recurrence theorem says that $X_n$ is recurrent in $d \leq 2$
  and transient in $d \geq 3$, while results of Grill \cite{grill}
  say that the centre-of-mass process $G_n$ is recurrent only in $d=1$
  and transient for $d \geq 2$. Thus the asymptotic behaviour of
  $X_n - G_n$ is not trivial; Theorem \ref{srwthm} says that it is recurrent
  if and only if $d \in \{1,2\}$.
  In particular
 when $d=2$, $X_n$ and $X_n - G_n$ are both
 recurrent, but $G_n$ is transient;
 see Figure \ref{fig2} for a simulation.

 \begin{rmk}
Theorem \ref{srwthm}   exhibits an amusing feature. With the notation $\Delta_n := X_{n+1} - X_n$
   it is not hard to see from (\ref{com}) that we may write (with $X_0 := \0$)
   \[ G_n = \sum_{i=0}^{n-1} \left( 1 - \frac{i}{n} \right) \Delta_i ; ~~~
   X_n - G_n =  \sum_{i=0}^{n-1} \left( \frac{i}{n} \right) \Delta_i .\]
   It follows that (for {\em fixed} $n$) $X_n-G_n$ and $G_n$ are very nearly {\em time-reversals} of each other:
   writing $\Delta'_i := \Delta_{n-i}$ we see that
   \[ X_n - G_n = \sum_{i=1}^{n} \left(1 - \frac{i}{n} \right) \Delta'_i .\]
   Despite this, the two {\em processes} behave very differently,
   as can be seen by contrasting Theorem \ref{srwthm} with Grill's result \cite{grill}.
 \end{rmk}

 It is natural to ask whether a continuous analogue of Theorem \ref{srwthm}
 holds. In the one-dimensional case, we would take $B_t$ to be standard Brownian motion
 and $G_t = t^{-1} \int_0^t B_s \ud s$, and ask about the joint behaviour of
 $(B_t , G_t)$; in higher dimensions, writing the $d$-dimensional Brownian motion
 as $(B^{(1)}_t, \ldots, B^{(d)}_t)$, the $i$th component $G^{(i)}_t$ of $G_t$ is
 $t^{-1} \int_0^t B^{(i)}_s \ud s$, and different components are independent.
 We could not find a Brownian analogue of Grill's theorem for (compact set)
 recurrence/transience of $G_t$ explicitly stated
 in the literature. The process $(t G_t)_{t \geq 0}$ is
 {\em integrated Brownian motion}, or the {\em Langevin process}, see e.g.\
 \cite{ac,iw} and references therein. The two-dimensional process
 $(B_t, tG_t)_{t \geq 0}$ is   the {\em Kolmogorov diffusion}
 \cite{iw}. Theorem \ref{srwthm} gives basic information about the
 joint behaviour of a discrete version of this process, under a re-scaling of the
 second coordinate.

\begin{figure}[h!]
\begin{center}
\includegraphics[angle=0, width=0.8\textwidth]{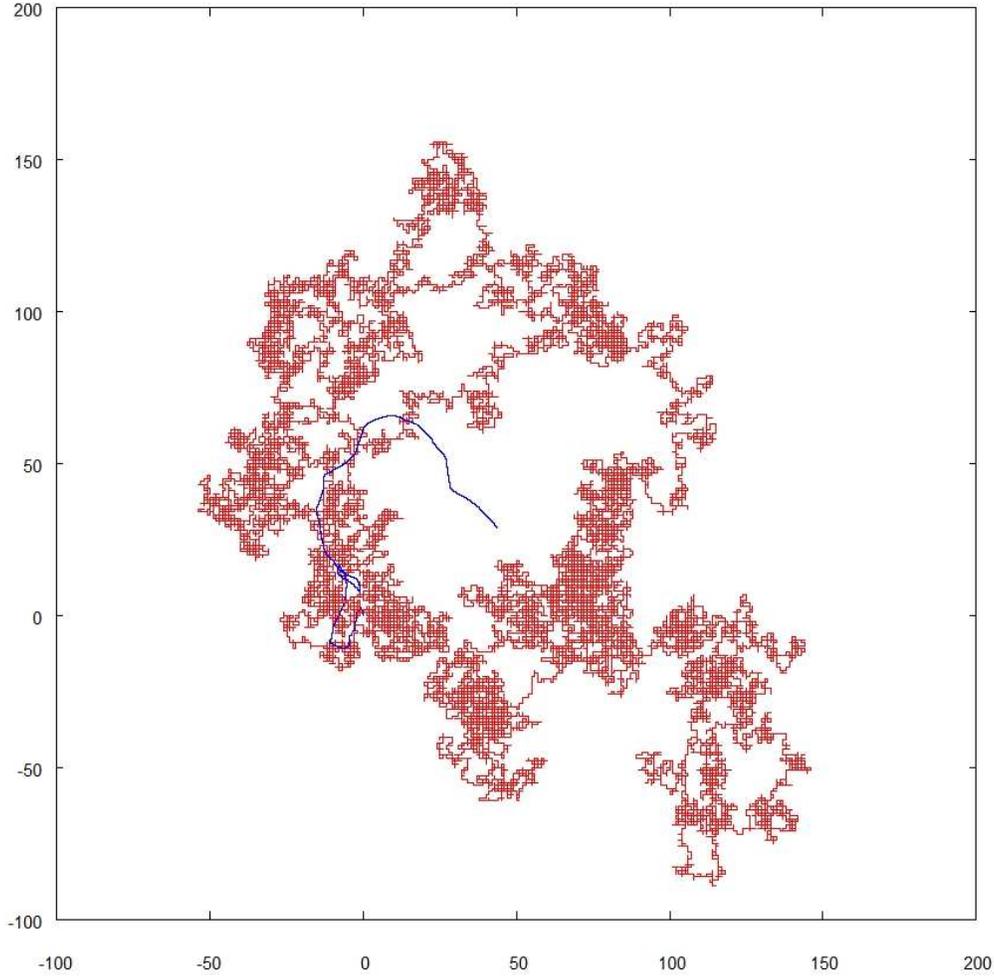}
\end{center}
\caption{Simulation of $4 \times 10^4$ steps of symmetric
SRW starting at the origin of $\Z^2$ (path shown in red) and its centre of mass
process (blue).}
\label{fig2}
\end{figure}

      In Section \ref{sec:model} we formally define our self-interacting random walk
  and state our main results.
  In Section \ref{sec:motiv} we discuss some more of the motivation
  behind
  our model (coming from the physics of polymers and also
  purely theoretical considerations)  and also the one-dimensional problems
  associated with (\ref{drift0a}) and (\ref{drift0b}), and
  explain how SRW (and Theorem \ref{srwthm})
  fits into our picture. The subsequent sections are devoted to the proofs.

  We finish this section with some comments
  on the relation of our model to the existing
  literature. We are not aware of any self-interacting
  random walk models similar to the one studied here (i.e.,
  interacting with the previous history of the process, as
  summarized through the barycentre).
  In broad outline, our model is related to
  the vertex-reinforced random walk (see \cite[Section 5.3]{pemantle})
  in that the evolution of the walk depends on the sites
  previously visited. A significant difference is that in vertex-reinforced random walk
  this self-interaction is local, in that only the occupation of
  nearest-neighbours of the current
  site affects the law of the increment, whereas our interaction,
  mediated by the barycentre, is global.
  In the continuous setting, self-interacting
  diffusions (or `Brownian polymers') with similar
  flavour and motivation to those of our model
  have also been studied over the last two decades or so, but
  are rather different in
  detail to the model considered here: see e.g.~\cite{nrw,dr,mt,blr} and
  references therein;
  some recent work on processes with self-attracting drift defined through a potential
 includes \cite{ck}.
   In the self-interacting
  diffusion setting, most of the results in the literature are concerned with the ergodic case;
  questions of recurrence/transience seem to have received little attention
  (particularly  in dimensions greater than 1), and we do not know of any
  results on asymptotic directions. Also, it is typically assumed that
  the vector consisting of the process and its empirical average are Markovian, whereas our model
  is more general. See \cite[Section 1]{mt} for a short survey.

 \section{The model and main results}
 \label{sec:model}

 \subsection{Definitions and assumptions}

We now define the stochastic
 process $X := (X_n)_{n \in \N}$ on $\R^d$ ($d \in \N$)
 that is our main object of study.
 (We start at time $n=1$ only so that (\ref{com}) has the neatest form.)
  The process $X$ will not be Markovian, as   the distribution of
  $X_{n+1}$ will depend on the entire history $X_1, \ldots, X_{n}$,
  although to a large extent this dependence will be mediated through
  the current centre of mass
 $G_n$ defined at (\ref{com}). Formally,
 we suppose that $(X_n)_{n \in \N}$ is adapted to the filtration
 $(\F_n)_{n \in \N}$;
 note that by (\ref{com}) $G_1,\ldots,G_n$ are $\F_n$-measurable.
  We use the notation $\Pr_n [ \, \cdot \,  ]  := \Pr [ \, \cdot \mid \F_n ]$
 and $\Exp_n[\, \cdot \, ] := \Exp [ \, \cdot \mid \F_n ]$.
 Throughout the paper we understand $\log x$ to mean $\log x$ if $x \geq 1$ and $0$ otherwise.

  We impose some specific assumptions on the law of
  $\Delta_n := X_{n+1} - X_n$ given $\F_n$.
We assume  that
for some $B \in (0,\infty)$ and all $n \in \N$,
  \begin{equation}
  \label{bound}
    \Pr_n [ \| \Delta_n \| \leq B   ] =1, \as.\end{equation}
    The assumption of uniformly bounded jumps can be replaced by an assumption on higher
    order moments at the expense of additional technical complications, but (\ref{bound})
    is  natural  when   the increments  represent  chemical bonds
    in a model for a polymer molecule.

        Our next assumption will be a precise
    version of (\ref{drift0}).
    We suppose that for some $\rho \in \R$ and $\beta \geq 0$,
    for any $n \in \N$,
       writing  $\bx = X_n - G_n $ for convenience,
 \begin{equation}
 \label{drift}
    \Exp_n [ \Delta_n  ] = \rho \| \bx \|^{-\beta} \hat \bx
    + O ( \| \bx \|^{-\beta} (\log \| \bx \|)^{-2} ), \as, \end{equation}
as $\| \bx \| \to \infty$,
 where $\hat \bx := \bx/\| \bx \|$.
 (In (\ref{drift}) the exponent $-2$ on the logarithm is chosen for simplicity;
 it could be replaced with any exponent strictly less than $-1$.)
 In equation (\ref{drift}) and similar vector equations in the sequel,
 terms such as $O( \, \cdot\,)$ indicate the presence of a vector
 whose norm satisfies the given $O(\, \cdot\,)$ asymptotics (similarly for $o(\,\cdot\,)$);
 error terms not involving $n$ are understood to be uniform in $n$. To be clear,
 (\ref{drift}) is to be understood as, with $X_n - G_n = \bx$,   as $\| \bx \| \to \infty$,
 \begin{align*}
  \sup_{n \in \N} \esssup  \| \Exp_n [ \Delta_n  ] - \rho \| \bx \|^{-\beta} \hat \bx \|
& = O ( \| \bx \|^{-\beta} (\log \| \bx \|)^{-2} ) . \end{align*}

  We also need to assume a uniform ellipticity condition, to ensure
  that our random walk does not get `trapped' in some subset of $\R^d$.
  Let $\Sp_d := \{ \be \in \R^d : \| \be \| =1 \}$ denote
  the unit-radius sphere in $\R^d$. We suppose that there
  exists $\eps_0>0$ such that
  \begin{align}
  \label{ue}
   \essinf_{\be \in \Sp_d}
    \Pr_n [ \Delta_n \cdot \be
    \geq \eps_0   ] \geq \eps_0.
  \end{align}

  Write $\Delta_n = (\Delta_n^{(1)}, \ldots, \Delta_n^{(d)})$
  in Cartesian components.
  An immediate consequence of (\ref{ue}) is the
  following   lower bound on  second moments: a.s.,
  \begin{equation}
  \label{var}
  \min_{i \in \{1,\ldots,d\}}
  \Exp_n [ (\Delta_n^{(i)})^2   ] \geq 2 \eps_0^3 > 0 .\end{equation}
 Our primary standing assumption will be the following.

   \begin{itemize}
   \item[(A1)] Let $d \in \N$. Let $X := (X_n)_{n \in \N}$
   be a stochastic process on $\R^d$ and $G:= (G_n)_{n \in \N}$
   its associated centre-of-mass process defined by (\ref{com}).
   For definiteness, take $X_1 \in \R^d$ to be fixed.
   Suppose that for some $B < \infty$, $\eps_0 >0$, $\rho \in \R$, and $\beta \geq 0$
   the conditions (\ref{bound}), (\ref{drift}), and (\ref{ue})  hold.
   \end{itemize}

In the examples discussed later (see Section \ref{sec:exam}), $(X_n,G_n)_{n \in \N}$
will be a Markov process, but we do not assume the Markov property in general.

   When $\beta =1$, as in the Lamperti case \cite{lamp1,lamp3}
   the value of $\rho$ in (\ref{drift})
  will turn out to be crucial. As in Lamperti's problem, the recurrence classification
   depends on the relationship between $\rho$ and the covariance structure
   of $\Delta_n$. To obtain an explicit criterion, we  impose additional regularity conditions on that covariance structure.
  Specifically, we sometimes
   suppose that (a)   there exists $\sigma^2 \in (0,\infty)$ such that,
    a.s.,
   \begin{equation}
   \label{cov1} \Exp_n [ ( \Delta_n^{(i)})^2  ] = \sigma^2 + o((\log \| X_n - G_n \| )^{-1}), ~~~( i \in \{1,\ldots,d\});\end{equation}
and (b) for $i, j$ distinct elements of $\{1,\ldots,d\}$,   a.s.,
   \begin{equation}
   \label{cov2} \Exp_n [ \Delta_n^{(i)} \Delta_n^{(j)} ] = o((\log \| X_n - G_n \| )^{-1}).\end{equation}
   Thus for $\beta \geq 1$,
   when necessary we will impose the following additional assumption.

    \begin{itemize}
   \item[(A2)] The conditions (\ref{cov1}) and (\ref{cov2})  hold for some $\sigma^2 \in (0,\infty)$.
   \end{itemize}

   \subsection{Results on self-interacting walk}
   \label{results}

   Our first result, Theorem \ref{ythm1}, constitutes the first part of our
   complete recurrence classification for $X_n - G_n$. Since we are dealing with non-Markovian
   processes, we first formally define what we mean by recurrence and transience in this context.

   \begin{df}
   \label{def1}
   An $\R^d$-valued stochastic process $(\xi_n)_{n \in \N}$ is said to be
   {\em recurrent} if $\liminf_{n \to \infty} \| \xi_n \| < \infty$ a.s.~and
   {\em transient} if $\lim_{n \to \infty} \| \xi_n \| = \infty$ a.s..
   \end{df}

      Define
  \begin{equation}
  \label{rho0}
   \rho_0 := \rho_0 (d,\sigma^2) := \frac{1}{2} (2-d) \sigma^2  .\end{equation}

  \begin{theo}
  \label{ythm1}
   Suppose that (A1) and (A2) hold
   with $d \in \N$, $\beta \geq 1$, and $\rho \in \R$.
   \begin{itemize}
   \item[(i)] Suppose that $\beta =1$.
  Let $\rho_0 = \rho_0 (d,\sigma^2)$ be as defined at (\ref{rho0}).
   Then $X_n - G_n$ is recurrent if $\rho \leq \rho_0$ and
     transient if $\rho > \rho_0$.
   \item[(ii)] Suppose that $\beta > 1$. Then $X_n - G_n$ is
   recurrent if $d \in \{1,2\}$ and transient if $d \geq 3$.
   \end{itemize}
  \end{theo}

 For almost all our remaining results we do not need to assume (A2).
     Set
    \begin{equation}
    \label{elldef}
   \ell (\rho, \beta) :=  \left(  \frac{\rho (1+\beta)}{2+\beta} \right)^{1/(1+\beta)} .
   \end{equation}
   In the case $\beta \in [0,1)$, we have the following result, which completes
   the recurrence classification for $X_n - G_n$ and also gives
   a   detailed account
    of the asymptotic behaviour of the random walk $X_n$.
   In particular, when $\rho>0$, $X_n$ and $G_n$ are transient, and moreover have a limiting direction,
   and the escape
   is quantified by super-diffusive but, for $\beta >0$, sub-ballistic strong laws of large numbers.
   The case $\beta=0$ shows ballistic behaviour.

  \begin{theo}
  \label{dirthm}
   Suppose that
   (A1) holds with $d \in \N$, $\beta \in [0,1)$, and $\rho \in \R \setminus \{0\}$.
   Then $X_n - G_n$ is
 transient if $\rho >0$ and recurrent if $\rho < 0$.
 Moreover, if $\rho >0$, there exists
 a random $\bu \in \Sp_d$ such that, as $n \to \infty$, with $\ell(\rho,\beta)$
 defined at (\ref{elldef}),
 \[ n^{-1/(1+\beta)}  X_n \toas ( 2+\beta )  \ell (\rho, \beta ) \bu,
 ~{\textrm{and}}~
 n^{-1/(1+\beta)}  G_n \toas ( 1+\beta )  \ell (\rho, \beta ) \bu.
  \]
  \end{theo}

At the level of detail displayed by Theorem \ref{dirthm}, we can see
a difference between the asymptotic behaviour of the $\beta \in [0,1)$,
$\rho>0$ case of (\ref{drift}) compared to the `supercritical
Lamperti-type' case in which the drift is away from a fixed origin
(i.e., the analogue of (\ref{drift}) holds but with $\bx = X_n$ rather
than $\bx = X_n - G_n$).
See Theorem \ref{dirthm0} below and the remarks that precede it.

   Our ultimate goal is a complete recurrence classification for   the process $X_n$.
   Theorem \ref{dirthm} covers the case $\beta \in [0,1)$, $\rho>0$.
   Otherwise,
  we have at the moment only the following one-dimensional result
  (to be viewed in conjunction with Theorem \ref{ythm1}).

  \begin{theo}
  \label{1dthm}
     Suppose that
   (A1) holds for $d=1$. Then if $X_n - G_n$ is transient,
  $X_n$ and $G_n$ are also transient,
  i.e., $|X_n| \to \infty$ and $|G_n| \to \infty$ a.s.~as
  $n \to \infty$.
  \end{theo}

 Our final result on our walk with barycentric interaction
 gives upper bounds on   $\|X_n\|$ for general $d \in \N$.
 In view
 of the interpretation of
 $(X_1,\ldots,X_n)$ as a model
 for a polymer molecule in solution, we can describe the phases
 listed in Theorem \ref{extent} below as
 (i) {\em extended},
 (ii) {\em transitional}, (iii) {\em partially collapsed},
and  (iv) {\em fully collapsed}.
 See the discussion in Section \ref{poly} below.
  Theorem \ref{extent}(i) is included for comparison only;
Theorem \ref{dirthm} gives a much sharper result.
Define
\begin{equation}
\label{gammadef}
 \gamma (d, \sigma^2 , \rho ) := \left( 2 - d -   \frac{2 \rho}{\sigma^2} \right)^{-1} .\end{equation}

 \begin{theo}
 \label{extent}
 Suppose that (A1) holds with $d \in \N$, $\beta \geq 0$, and $\rho \in \R$.
 Then the following bounds apply.
 \begin{itemize}
 \item[(i)] (Theorem \ref{dirthm}.) If $\beta \in [0,1)$ and $\rho>0$,
 there exists $C \in (0,\infty)$ such that, a.s.,
  $\| X_n \| \leq C n^{1/(1+\beta)}$ for all but finitely many $n \in \N$.
  \item[(ii)]  If $\beta \geq 1$,  then for any $\eps>0$, a.s.,
  $\| X_n \| \leq  n^{1/2}( \log n)^{(1/2)+\eps}$ for all but finitely many $n \in \N$.
  \item[(iii)] Suppose that (A2) also holds. Suppose that $\beta =1$ and $\rho < -d\sigma^2/2$, and
  let $\gamma(d,\sigma^2,\rho) \in (0,1/2)$ be as defined at (\ref{gammadef}). Then
  for any $\eps>0$, a.s., for all but finitely many $n \in \N$,
  $\| X_n \| \leq n^{\gamma(d,\sigma^2,\rho)+\eps}$.
    \item[(iv)]
  If $\beta \in [0,1)$ and $\rho <0$,
  then for any $\eps>0$, a.s.,
  $\| X_n \| \leq    (\log n)^{1+\frac{1}{1-\beta} +\eps}$ for all but finitely many $n \in \N$.
  \end{itemize}
  \end{theo}

 We suspect that the bounds in  Theorem \ref{extent} are close to sharp, in that
corresponding lower bounds of almost the same order should be valid (only infinitely often, of course, in the recurrent cases).
However, the lower bounds of \cite[Section 4]{mvw} do not apply directly.

Given (\ref{com}) it is evident that
  the bounds for $\|X_n \|$ in Theorem \ref{extent}
  imply the same bounds (up to multiplication by a constant) for $\| G_n \|$, and hence
  $\| X_n - G_n \|$ too. In addition, the same upper
  bounds hold (again up to a constant factor)
  for the quantities
  of {\em diameter} $D_n$ and root-mean-square {\em radius of gyration} $R_n$
    given by
  \[ D_n := \max_{1 \leq i < j \leq n } \| X_i - X_j \| , ~~~
    R_n^2 :=  \frac{1}{n} \sum_{i=1}^n \| X_i - G_n \|^2
  = \frac{1}{n^2} \sum_{i=1}^n \sum_{j=1}^{i-1} \| X_i - X_j \|^2 ;\]
  these are both physically significant in the interpretation
  of $(X_1,\ldots,X_n)$ as a polymer chain (see pp.~95--96 of \cite{hughes} and
Section \ref{poly} below).

Finally, we briefly describe how our results compare to
the more classical model studied by Lamperti \cite{lamp1,lamp3}.
That is, suppose that (A1) holds but that
(\ref{drift}) holds with $\bx = X_n$
{\em instead of} $\bx = X_n - G_n$. In this case,
there is no self-interaction in the drift term,
and the drift is relative to a fixed origin. Lamperti
studied examples  of such processes (so-called centrally biased random walks) in \cite[Section 4]{lamp1}
and \cite[Section 5]{lamp3}. We see that our
recurrence classification for the self-interacting process
 $X_n - G_n$ in the case $\beta=1$,
Theorem \ref{ythm1}, gives, surprisingly,
essentially the same criteria as Lamperti's
 \cite[Theorem 4.1]{lamp1}. In the case $\beta \in [0,1)$,
 the
difference between the two settings is clearly
manifest in the constant in the law of large numbers.
The analogue of our Theorem \ref{dirthm}
in the case of drift relative to the origin is an immediate
consequence of Theorem 2.2 of \cite{mmw}
with Theorem 3.2 of \cite{superlamp} (see the discussion in
\cite[Section 3.2]{superlamp}):
\begin{theo}
\label{dirthm0}
\cite{mmw,superlamp}
  Suppose that
   (A1) holds, with the modification that
   (\ref{drift}) holds with $\bx =X_n$
instead of $\bx = X_n - G_n$.
Suppose that $d \in \N$, $\beta \in [0,1)$, and $\rho >0$.
  Then there exists
 a random $\bu \in \Sp_d$ such that, as $n \to \infty$,
 \[ n^{-1/(1+\beta)}  X_n \toas (2+\beta)^{1/(1+\beta)} \ell (\rho, \beta) \bu,
 ~{\textrm{and}}~
 n^{-1/(1+\beta)}  G_n \toas ( 1+\beta ) (2+\beta)^{-\beta/(1+\beta)} \ell (\rho, \beta ) \bu.
  \]
\end{theo}
The method of proof of Theorem \ref{dirthm} in the
present paper (see Section \ref{direction})
gives an alternative proof of Theorem \ref{dirthm0},
avoiding the rather involved argument
for establishing a limiting direction used in \cite{mmw}.
Specifically, in the argument in Section \ref{direction},
we can apply the relevant law of large numbers (Theorem 3.2
of \cite{superlamp}) in place of our Lemma \ref{ylln}.
Note that under the assumption of bounded increments, the law of large
numbers \cite[Theorem 3.2]{superlamp} is available, unlike in the generality of Theorem 2.2
from \cite{mmw}; thus in the more general setting, the proof
of \cite{mmw} is currently the only one that the authors are aware of.

  \subsection{Examples}
  \label{sec:exam}

  To illustrate our assumptions and results, we give three examples
  of   walks satisfying (A1) and (A2). In all of the following examples,
  the couple $(X_n,G_n)$ is Markov.

  \paragraph{Example 1.} For $\bx \in \R^d$,
  let $\bb_1 (\bx), \ldots, \bb_d(\bx)$ denote
  an orthonormal basis for $\R^d$ such that $\bb_1(\bx) = \hat \bx$, the unit
  vector in the direction $\bx$; we use the convention $\hat \0 := \be_1 := (1,0,\ldots,0)$.
  Fix $\eps_0 \in (0, 1/(2d))$, $\rho \in \R$, and $\beta >0$.
  Take
  \[ \Pr_n [ \Delta_n = \bb_i (X_n-G_n) ] = \Pr_n [ \Delta_n = -\bb_i (X_n-G_n) ] = \frac{1}{2d}, ~~~(i \in \{2,\ldots,d\} ),\]
  and
  \begin{align*} \Pr_n [ \Delta_n = \bb_1 (X_n-G_n) ] & =
  \begin{cases}
  \frac{1}{2d} +
  \frac{\rho}{2} \| X_n - G_n \|^{-\beta} & {\rm if} ~\frac{| \rho |}{2} \| X_n - G_n \|^{-\beta} \leq \frac{1}{2d} -\eps_0 \\
  \frac{1}{d} - \eps_0 & {\rm if} ~
      \frac{ \rho }{2} \| X_n - G_n \|^{-\beta} > \frac{1}{2d} -\eps_0 \\
         \eps_0 & {\rm if} ~
      \frac{ \rho }{2} \| X_n - G_n \|^{-\beta} < - \frac{1}{2d} +\eps_0
      \end{cases}; \\
       \Pr_n [ \Delta_n = -\bb_1 (X_n-G_n) ] & = \frac{1}{d} - \Pr_n [ \Delta_n = \bb_1 (X_n-G_n) ].
  \end{align*}
  In other words, for all $\| X_n - G_n \|$ sufficiently large,
  \[ \Pr_n [ \Delta_n = \pm \bb_1 (X_n-G_n) ]   =
   \frac{1}{2d} \pm
  \frac{\rho}{2} \| X_n - G_n \|^{-\beta} .\]
  Then writing $\bx = X_n - G_n$, we have for $\bx \in \R^d$ with $\| \bx \|$ sufficiently large, a.s.,
  \[
  \Exp_n [ \Delta_n  ] = \rho \| \bx\|^{-\beta} \hat \bx;
  ~~~ \Exp_n [ (\Delta_n^{(i)})^2 ] = \frac{1}{d} \sum_{j=1}^d ( \bb_j \cdot \be_i )^2 = \frac{1}{d} .\]
  It is not hard to verify that (A1) and (A2) (with $\sigma^2 = 1/d$)
  hold in this case.
 In particular, if $\beta =1$ Theorem \ref{ythm1} says that $X_n-G_n$
 is transient if and only if $\rho > (2-d)/(2d)$. See Figure \ref{fig3}
 for some simulations of this model.

  \begin{figure}[h!]
\begin{center}
\includegraphics[angle=0, width=0.48\textwidth]{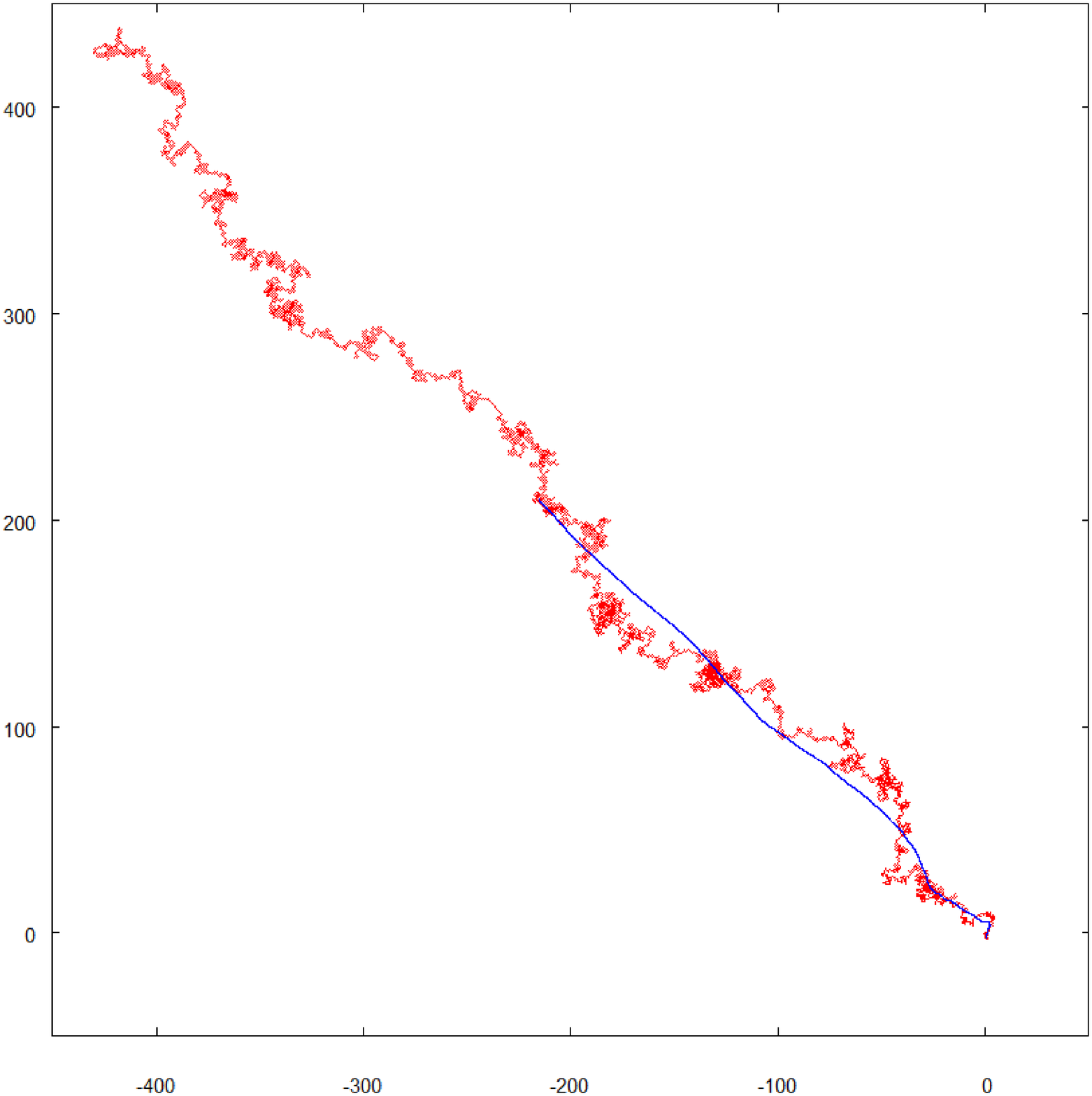}
\includegraphics[angle=0, width=0.48\textwidth]{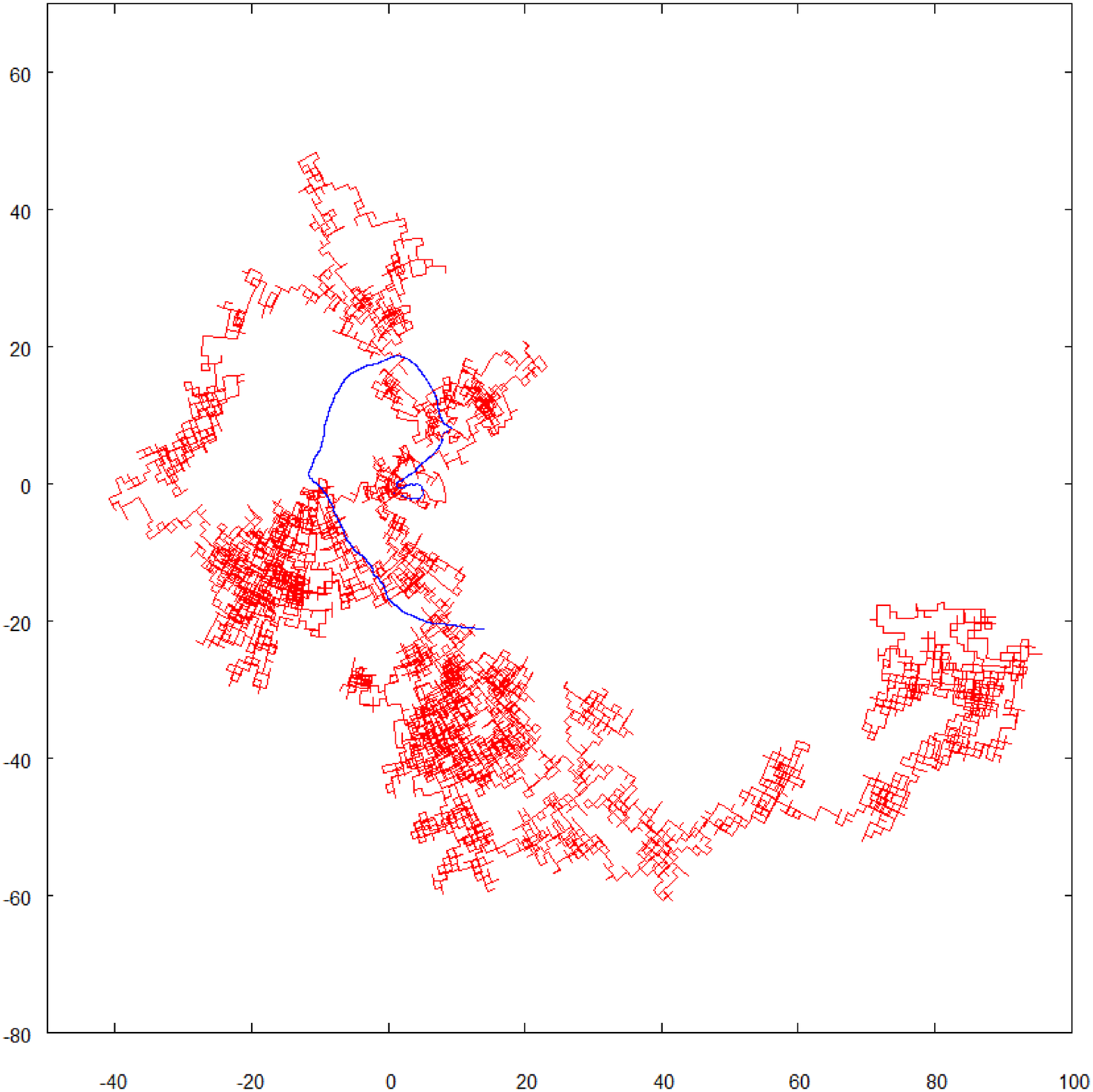}
\includegraphics[angle=0, width=0.48\textwidth]{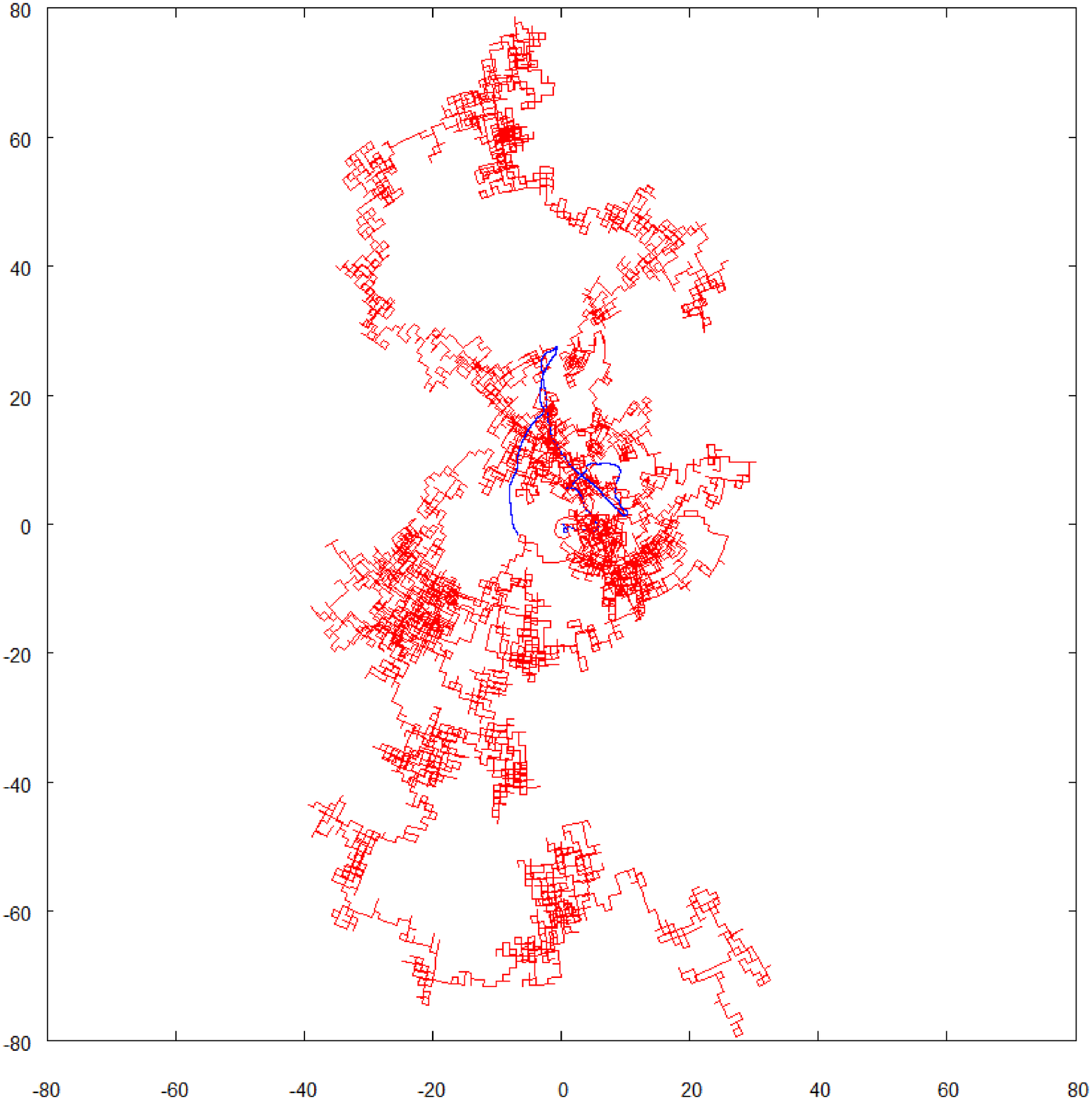}
\end{center}
\caption{Simulation of $10^4$ steps of the random walk (red) and its centre
of mass (blue), as described in Example 1, with $d=2$, $\rho =0.1$, $\eps_0 = 0.01$,
and different values of $\beta \in (0,1]$; the three pictures have $\beta = 0.1, 0.5, 1$
clockwise from the top left. Theorem \ref{dirthm} says that in the two $\beta <1$ cases,
the random walk $X_n$ is transient with a limiting direction. In the $\beta=1$ case, we know
from Theorem \ref{ythm1} that $X_n - G_n$ is transient ($\rho_0 = 0$ here),
 but transience of $X_n$ itself
is an open problem.}
\label{fig3}
\end{figure}

\paragraph{Example 2.} Here is another example satisfying (A1) and (A2), this time with jumps
supported on a unit sphere rather than being restricted to a finite set of possibilities.
Let $\beta >0$ and $\rho \in \R$.
Given $\F_n$ and $X_n - G_n = \bx$, the jump $\Delta_n$ is obtained as follows.
\begin{itemize}
\item[(i)] Choose ${\bf U}_n$ uniformly distributed on the unit sphere $\Sp_d$.
\item[(ii)] Take $\Delta_n = {\bf U}_n  + \rho \| \bx \|^{-\beta} \1_{\{ \rho \| \bx \|^{-\beta} < 1/2 \}} \hat \bx $.
\end{itemize}
So the jumps of the walk are uniform on a sphere, but the centre of the sphere is (for $\| \bx \|$ large enough)
shifted slightly in the direction $\pm \hat \bx$, depending on the sign of $\rho$.
Conditions (A1) and (A2) (again with $\sigma^2=1/d$)
are readily verified for this example.

\paragraph{Example 3.} We sketch one more example with $d \geq 2$, $\beta >0$ and
$\rho>0$, which is reminiscent of the walk avoiding its convex hull.
Take the jump $\Delta_n$ uniform on  $\Sp_d$ minus the
circular cap of relative surface $\rho \| X_n - G_n \|^{-\beta}$ pointing towards
the barycentre, i.e.,
$$ \Delta_n ~\text{is uniform on}~ \left\{\by \in \Sp_d : \by \cdot \hat \bx > -1 + C(\rho) \|\bx\|^{-\beta/(d-1)}
\right\}, $$
 with $\bx=X_n-G_n$,
where $C(\rho)$ is a constant depending on $\rho$ and $d$.
Here we assume
$\|\bx\|$ is sufficiently large; if not we can take $\Delta_n$
uniform on $\Sp_d$.

   \subsection{Open problems and paper outline}

   Our results  give a detailed
   recurrence classification (Theorem \ref{ythm1}) for the process $X_n - G_n$.
   Of considerable interest is the asymptotic behaviour of $X_n$ itself,
   for which we have a complete picture only in the case $\beta \in [0,1)$, $\rho>0$ (Theorem \ref{dirthm}).
   We conjecture:
   \begin{itemize}
   \item $\|X_n \| \to \infty$ a.s.\ if and only if $\| X_n - G_n \| \to \infty$ a.s..
  \end{itemize}
  Theorems \ref{dirthm} and \ref{1dthm} verify the `if' part of the conjecture when
  (i) $\beta \in [0,1)$ and $\rho>0$, and (ii) $d=1$. Another open problem involves the angular
  behaviour of our model when $\beta \geq 1$. By analogy with \cite{mmw}
  we suspect that there is {\em no} limiting direction in that case
  (in contrast to Theorem \ref{dirthm}).

  The remainder of the paper is arranged as follows. In Section \ref{sec:motiv}
  we describe in more detail how our model is related to Lamperti's
  problem (Section \ref{sec:lampsrw}), and to the centre-of-mass of SRW (Section \ref{sec:com}),
  and we prove Theorem \ref{srwthm}. We also outline the
  motivation of our random walk as a model for a random polymer
  in solution (Section \ref{poly}). Section \ref{prelim}
  is devoted to preliminary computations for the processes
  $X_n$, $G_n$, and (especially) $X_n - G_n$. In Section
  \ref{1dproc} we take a somewhat more general view,
  and study the asymptotic properties of one-dimensional,
  not necessarily Markov, processes satisfying a precise version of
  (\ref{drift0b}). The recurrence classification
  is a time-varying, more complicated analogue of Lamperti's results \cite{lamp1,lamp3},
  and we use some martingale ideas related to those in \cite{FMM,superlamp}. In the case $\beta \in [0,1)$,
  $\rho>0$ we prove a law of large numbers that is a cornerstone of our subsequent analysis
  for the random walk $X_n$. This law of large numbers
  is an analogue of that in \cite{superlamp}
  for the supercritical Lamperti problem. While the results of \cite{superlamp}
  supply an upper bound crucial to our approach, the law of large numbers in the present setting requires a new idea,
  and our key tool here is a stochastic approximation  lemma (Lemma \ref{lower-lm}), which may be of independent interest.
    Section \ref{proofs}
  is devoted to the proofs of our main theorems.
  The basic method is an application of the results
  of Section \ref{1dproc} to the process $\| X_n - G_n \|$, armed
  with our computations in Section \ref{prelim}.
  We carry out this approach to prove Theorems \ref{ythm1} and \ref{1dthm} in Section \ref{recprf}.
  A crucial ingredient
  is the proof, in Section \ref{direction}, that $X_n - G_n$ has a limiting
  direction. This enables us to prove Theorem \ref{dirthm}.
  Finally, in Section \ref{boundprf}, we
   prove Theorem \ref{extent}, building on some general results
  from \cite{mvw}.

    \section{Connections and further motivation}
  \label{sec:motiv}

  \subsection{Lamperti's problem and simple random walk norms}
  \label{sec:lampsrw}

  Our problem is related to a time-dependent
  version of the so-called Lamperti problem.
  We briefly review the latter here.
  Let $Z = (Z_n)_{n \in \N}$  be a discrete-time
 stochastic process adapted to a filtration
  $(\F_n)_{n \in \N}$   and
taking values in an unbounded
subset $\SS$ of $[0,\infty)$. The set
$\SS$ may be countable (as in the SRW example which follows in this section) or uncountable
(as in the application to stochastic
billiards described in \cite{mvw}).

Lamperti \cite{lamp1,lamp2,lamp3}
investigated the extent to which
the asymptotic
behaviour of $Z$ is determined by the
increment moments $\Exp_n [ (Z_{n+1} - Z_n)^k ]$
when viewed as (random) functions of $Z_n$.
Formally, suppose that
for some $k$,
$\Exp_n [ (Z_{n+1} - Z_n)^k ]$ is well-defined for all $n$.
Then
by standard properties of conditional expectations
(see e.g.~\cite[Section 9.1]{chung}),   there exist a Borel-measurable
function $\phi_k(n ; \, \cdot \,)$ and an $\F_n$-measurable
random variable $\psi_k (n)$ (orthogonal to $Z_n$) such that, a.s.,
\[  \Exp_n [ (Z_{n+1} - Z_n)^k ] =  \Exp  [ (Z_{n+1} - Z_n)^k \mid Z_n ]   + \psi_k (n)
= \phi_k (n ; Z_n ) + \psi_k (n) .\]
Define
\begin{equation}
\label{mudef}
 \mu_k (n ; x ) := \phi_k (n ; x) + \psi_k (n) .\end{equation}
The $\mu_k (n ; x )$ are, in general, $\F_n$-measurable
random variables; if $Z$ is a Markov process
then $\mu_k (n ; x) = \Exp  [ (Z_{n+1} - Z_n)^k \mid Z_n=x ]$ is a deterministic
function of $x$ and $n$, and if in addition $Z$ is
time-homogeneous, $\mu_k (n ; x) = \mu_k (x)$
is a function of $x$ only.
  For many applications,
  including those described here, $Z$ will not be
  time-homogeneous or Markovian, but nevertheless the
  $\mu_k (n ;x)$
  are well-behaved asymptotically.

In this section, $X = (X_n)_{n \in \N}$ will be the symmetric
SRW on $\Z^d$ ($d \in \N$). That is,
$X$ has i.i.d.~increments $\Delta_n := X_{n+1} -X_n$ such that
 if $\{ \be_1,\ldots,\be_d\}$
is the standard orthonormal basis on $\R^d$, for $i \in \{1,\ldots,d\}$,
$\Pr [ \Delta_n = \be_i ] = \Pr [ \Delta_n = - \be_i ] = (2d)^{-1}$.

Let $\F_n = \sigma (X_1,\ldots,X_n)$ and
consider the $(\F_n)_{n \in \N}$-adapted
process $Z=(Z_n)_{n \in \N}$
on $[0,\infty)$ defined
by $Z_n = \| X_n \|$.
Here $Z$ takes values in
the countable
set $\SS = \{ \| \bx \| : \bx \in \Z^d \}$.
Note that $Z$ is not in general
a Markov process:
when $d=2$,
given one of the two $\F_n$-events
$\{X_n = (5,0)\}$
and $\{X_n = (3,4)\}$
we have $Z_n = 5$ in
 each case
but $Z_{n+1}$ has two different distributions;
for instance $Z_{n+1}$ can take the value $6$
(with probability $1/4$) in the
first case, but this is impossible in the
second case.

We recall some simple facts about $\Delta_n = X_{n+1} - X_n$ in the case of SRW.
We have \begin{equation}
  \label{bound1}
   \Pr_n [ \| \Delta_n \| \leq 1  ] =1 , \as,
   ~~{\rm and} ~~
   \Exp_n  [ \Delta_n ] =\0, \as.
   \end{equation}
  Writing $\Delta_n = (\Delta_n^{(1)}, \ldots, \Delta_n^{(d)})$
  in Cartesian components, we have that
  \begin{equation}
  \label{mom2}
   \Exp_n  [  \Delta_n^{(i)}  \Delta_n^{(j)}   ] = \frac{1}{d} \1 \{ i = j\} , \as. \end{equation}
Elementary calculations based on Taylor's expansion and (\ref{bound1}) and (\ref{mom2})
show that
\begin{align*}
\Exp_n  [ Z_{n+1} -Z_n   ]
= \frac{1}{2d} \sum_{i=1}^d \left( \| X_n + \be_i \| + \| X_n - \be_i \| - 2\| X_n \| \right) \\
= \frac{1}{2 \| X_n \|} \left( 1 - \frac{1}{d} \right) + O ( \| X_n \|^{-2} ) ;
\end{align*}
in the above notation,
$\mu_1 ( n ; x) = \frac{1}{2x} (1 - \frac{1}{d} ) + O(x^{-2})$ as $x \to \infty$. As before, this asymptotic
expression is the compact notation for
\[ \sup_{n \in \N} \esssup \mu_1 (n ; x) = \frac{1}{2x} \left(1 - \frac{1}{d} \right) + O(x^{-2}), \]
together with the same expression with `$\inf$' instead of each `$\sup$'.
Similarly
\begin{align*}
\Exp_n  [ Z^2_{n+1} -Z^2_n   ]
= \frac{1}{2d} \sum_{i=1}^d \left( \| X_n + \be_i \|^2 + \| X_n - \be_i \|^2 - 2\| X_n \|^2 \right)
= 1. \end{align*}
Then since $(Z_{n+1}-Z_n)^2 = Z_{n+1}^2 -Z_n^2 - 2 Z_n (Z_{n+1}-Z_n)$ we obtain
\[ \Exp_n  [ (Z_{n+1} -Z_n)^2   ]
= \frac{1}{d} + O (\| X_n \|^{-1} ) .\]
 In particular, (\ref{drift0a}) holds
 (interpreted correctly)
 with $\beta =1$ and $\rho' = (1 - (1/d))/2$.

    \subsection{Centre of mass for simple random walk}
   \label{sec:com}

  We saw in Section \ref{sec:lampsrw} how a  mean drift
  described loosely by (\ref{drift0a}) arises from
   the process of norms of symmetric SRW.
   In this section we describe how a process with mean
   drift of the form (\ref{drift0b}) arises when considering
   the distance of a symmetric  SRW
   to its centre of mass.
         The motion of the centre of mass of a random walk
   is of interest from a physical point of view, when,
   for example, the  walk represents a growing
   polymer molecule: see e.g.\ \cite{as} and \cite{rg}, especially Chapter 6.

   The centre-of-mass process (defined by (\ref{com}))
   corresponding to a symmetric SRW
   on $\Z^d$ was studied by Grill \cite{grill},  who showed that the process
   $(G_n)_{n \in \N}$ returns to a fixed ball containing
   the origin with probability $1$ if and only if $d=1$. In particular
   the process is transient for $d \geq 2$ and
   Grill gives a sharp integral test for the rate of escape
   of the lower envelope. A consequence of his result is the following.

   \begin{theo}
   \label{grill}
   \cite{grill} Let $(X_n)_{n \in \N}$ be symmetric SRW
   on $\Z^d$ and $(G_n)_{n \in \N}$ the corresponding
   centre-of-mass process defined by (\ref{com}).
Let $d \in \{ 2,3,4,\ldots\}$. Then for any $\eps>0$,
   \[ \| G_n \| \geq ( \log n) ^{-\frac{1}{d-1} - \eps} n^{1/2}, \as, \]
   for all but finitely many $n \in \N$. On the other hand, for infinitely many $n \in \N$,
    \[ \| G_n \| \leq ( \log n) ^{-\frac{1}{d-1}} n^{1/2}, \as. \]
    \end{theo}
      A crude upper bound for
  $\| G_n\|$, obtained by applying the triangle inequality
$\| G_n \| \leq \frac{1}{n} \sum_{i=1}^n \| X_i \|$
  and the
   law
  of the iterated logarithm for symmetric SRW on $\Z^d$ ($d \in \N$)
  to each $\| X_i \|$
  (see e.g.~Theorem 19.1 of \cite{rev}), is that for any $\eps>0$,
  a.s.,
 \[
   \| G_n \| \leq \frac{2}{3d^{1/2}} (1+\eps) (2 n \log \log n)^{1/2} , \]
   for all but finitely many $n\in \N$;
 it seems likely that this is an overestimate.
In $d=1$, in the analogous continuous setting, a result of Watanabe
\cite[Corollary 1, p.~237]{wat}
says that, for $B_t$ standard Brownian motion, for any $\eps>0$, for all $t$ large enough,
\[ \frac{1}{t}   \int_0^t B_s \ud s   \leq 3^{-1/2} (1+\eps) (2 t \log \log t)^{1/2}, \as, \]
  and this bound is sharp in that
the inequality fails infinitely often, a.s., when $\eps=0$. Standard strong approximation results
show that this result can be transferred to $\| G_t \|$ in $d=1$.

The next result shows how the drift equation (\ref{drift0b}) arises
in this context. Lemma \ref{srwlem} is a consequence
of the more general
Lemma \ref{lem1} below.

  \begin{lm}
  \label{srwlem}
  Let $d \in \N$.
  Suppose that $(X_n)_{n \in \N}$ is a symmetric SRW on $\Z^d$, and
  $(G_n)_{n \in \N}$ is its centre-of-mass process
  as defined by (\ref{com}). Let $\F_n := \sigma (X_1, \ldots, X_n)$ and
   $Y_n := X_n - G_n$.
  Then, a.s.,
      \begin{align*}
   \Exp_n [\| Y_{n+1} \| - \| Y_n \|   ]
 & = \left(1- \frac{1}{d} \right) \frac{1}{2\| Y_n \|} - \frac{ \| Y_n \|}{n+1}
  + O( \| Y_n \|^{-2} ) ; \\
\Exp_n [ (\| Y_{n+1} \| - \| Y_n \| )^2   ]
  & =   \frac{1}{d} + O(  \| Y_n \| n^{-1} ) + O( \| Y_n \|^{-1} ).
    \end{align*}
   \end{lm}
    Neglecting   higher-order terms, the study of the process $\| X_n - G_n\|$ for SRW leads
    us to analysis of a process with drift given by (\ref{drift0b}).
    Lemma \ref{srwlem} can be generalized to zero-drift
    Markov chains $X=(X_n)_{n \in \N}$ satisfying
    appropriate versions of (\ref{bound1}) and (\ref{mom2}).
  We prove our results on SRW by applying our general results
 given in Sections \ref{prelim} and \ref{1dproc}.

 \begin{proof}[Proof of Lemma \ref{srwlem}.]
  This follows from Lemma \ref{lem1} stated and proved in Section \ref{prelim}.
     Taking expectations in (\ref{lem1eq}) and using
        (\ref{bound1}) and (\ref{mom2}) we obtain the first
        equation in the statement of the lemma,
        using the fact that $\| Y_n \| = o(n)$ a.s.\ to simplify the error terms.
         Similarly, squaring both sides of (\ref{lem1eq}) and taking expectations we obtain
 the second equation in the lemma. \end{proof}

 \begin{proof}[Proof of Theorem \ref{srwthm}.]
     Let $Z_n = \| Y_n \| = \| X_n - G_n\|$ and $\F_n = \sigma ( X_1, \ldots, X_n)$.
   Then by Lemma \ref{srwlem}, a.s.,
   \begin{align*}
    \Exp_n [ Z_{n+1} - Z_n  ] & = \left(1 - \frac{1}{d} \right) \frac{1}{2Z_n} - \frac{Z_n}{n}
   + O(n^{-2} Z_n) + O( Z_n^{-2} ) ; \\
    \Exp_n [ (Z_{n+1} - Z_n)^2 ] & =  \frac{1}{d} + O( Z_n n^{-1} ) + O(Z_n^{-1} ) .\end{align*}
   Thus (\ref{z2b}) and (\ref{z3b}) hold with $\rho' = (d-1)/(2d)$ and $\sigma^2 = 1/d$.
 It follows from  Theorems \ref{thm1} and \ref{thm2} (stated and proved in Section \ref{1dproc})
   that $Z_n$ is transient if and only if $2 \rho' > \sigma^2$, or equivalently
   $1 - (1/d) > (1/d)$,
   that is, $d>2$. \end{proof}

    \subsection{The process viewed as a new random polymer model}
 \label{poly}

 In this section we briefly summarize motivation
 of self-interacting random walks arising from
 polymer physics, and give an interpretation
 of our model described by (A1) in that context.
 Much more background is provided
 by, for instance, \cite[Section 2.2]{ms}, \cite[Chapter 7]{rg}, \cite[Chapter 7]{hughes},
 and, for the underlying physics, \cite{rc}.
Recent accounts of some of the relevant
probability theory are given in \cite{giac,holl}.

 The sites visited by the walk $X_n$ represent the monomers
 that make up a long polymer molecule in solution in $\R^d$
 (of course, physically $d \in \{2,3\}$ are most interesting).
 The line segments between successive sites $X_n$ and $X_{n+1}$
 represent the chemical bonds holding the molecule together;
 in this regard our condition of uniformly bounded increments
 in (A1) is natural. We assume that the polymer solution
 is dilute, so that interaction between different polymer molecules
 can be neglected.

 In real polymers, a phase transition is observed
  between polymers in {\em poor solvents}
  (or at low temperature) and {\em good solvents}
  (or high temperature) \cite[Chapter 7]{rc}.
  In poor solvents, a polymer molecule collapses as the attraction between
  monomers overcomes the excluded volume effect caused by the fact that no
  two monomers can occupy the same physical space. In good solvents,
  a polymer molecule exists in an extended phase where the excluded volume
  effect dominates.

  It is the extended phase that is believed to lie
  in the same universality class as SAW. Heuristic arguments
  dating back to P.J.~Flory (see e.g.~\cite[Section 2.2]{ms})
  suggest that in this phase $\| X_n \|$
  should exist on `macroscopic scale' of order $n^\nu$ for an
  exponent $\nu = \nu(d) \in [1/2,1]$, with $\nu <1$ for $d>1$ and
  $\nu > 1/2$ for $d \leq 3$. So for $d \in \{2,3\}$, the polymer
  is expected to be super-diffusive but sub-ballistic.
  According to Theorem \ref{extent}(i),
  our model defined by (A1) has macroscopic
  scale exponent $\max \{ 1/2, 1/(1+\beta) \}$ when $\rho >0$; for $\beta < 1$
  this regime therefore corresponds to polymers in the extended phase,
  where the excluded volume effect, summarized by repulsion from the centre of
  mass, dominates.
  For instance, since $\nu(2) =3/4$, in $d=2$ the `physical'
  choice of our model has $\beta =1/3$ and $\rho >0$; it is not clear
  to what extent that case of our model replicates the behaviour of SAW.

  On the other hand, the collapsed phase corresponds
  to taking $\rho <0$ in (A1), where the polymer's
  self-attraction, summarized through its centre of mass,
  is dominant. See Theorem \ref{extent}(iii) and (iv).
  Between the poor and good solvent phases, there is
  a transitional phase at the so-called
  $\theta$-{\em point} at which the temperature achieves a specific (critical) value $T = \theta$.
  Here the excluded volume effect and self-attraction
  are in balance, and the molecule behaves
  rather like a simple random walk path.
 Compare Theorem \ref{extent}(ii).

 \section{Properties of the self-interacting random walk}
 \label{prelim}

 Under the assumption (A1), we
 are going to study the process $X_n - G_n$ and in particular
 determine whether it is transient or recurrent.
 It   suffices to study   $\| X_n - G_n \|$.
 In this section we analyse  the basic
 properties of the latter process; subsequently we will
  apply our general results of Section \ref{1dproc}
 on processes that satisfy, roughly speaking, (\ref{drift0b}).

 We introduce some convenient notation that we use throughout.
 For $n \in \N$ set
 \[ Y_n := X_n - G_n, ~~~\Delta_n := X_{n+1} - X_n .\]
  We start with some elementary relations amongst $X_n$, $G_n$,
 and $Y_n$ following from  (\ref{com}).

  \begin{lm}  Suppose that $(X_n)_{n \in \N}$ is a stochastic process on $\R^d$, and
  $(G_n)_{n \in \N}$ is its centre-of-mass process
  as defined by (\ref{com}).
  For $n \in \N$ we have
   \begin{align}
  \label{gjump}
   G_{n+1} & = \frac{n}{n+1} G_n + \frac{1}{n+1} X_{n+1} ;  ~\textrm{and} \\
     \label{yeq}
   Y_{n+1} & =   \frac{n}{n+1} ( Y_n + \Delta_n  ) .\end{align}
Moreover
  $G_1 = X_1$ and for $n \in \{2,3,\ldots\}$,
  \begin{equation}
  \label{gandy}
   G_n = X_1 + \sum_{j=2}^n \frac{1}{j-1} Y_j .\end{equation}
  \end{lm}
  \begin{proof}
Equation (\ref{gjump}) is immediate from (\ref{com}).
Then from (\ref{gjump}) we have that for $n \in \N$,
  \begin{align}
  \label{jj0}
  Y_{n+1}  = X_{n+1} - G_{n+1} = \frac{n}{n+1} \left( X_{n+1} - G_n \right) ,
  \end{align}
  from which (\ref{yeq}) follows since  $X_{n+1} - G_n = Y_n + \Delta_n$.
  For (\ref{gandy}), we have from (\ref{gjump}) again that for $n \in \N$,
  \[ G_{n+1} - G_n = \frac{1}{n+1} \left( X_{n+1} - G_n \right) = \frac{1}{n} Y_{n+1} ,\]
  where the final equality is obtained from (\ref{jj0}).
  Thus for $n \geq 2$,
  \[ G_n - G_1 = \sum_{j=1}^{n-1} (G_{j+1} - G_j ) = \sum_{j=1}^{n-1} \frac{1}{j} Y_{j+1} ,\]
  from which (\ref{gandy}) follows. \end{proof}

 The main result of this section concerns
 the increments of the process
 $\| Y_n \|$ under assumption (A1)
 and also possibly (A2).
 Part (i) of Proposition \ref{propinc}
 gives basic regularity properties,
 including boundedness of jumps. Part (ii)
 gives an expression for the mean drift
 when $\beta \in [0,1)$. Part (iii)
 deals with the case $\beta \geq 1$ when (A2) also holds.

   \begin{proposition}
   \label{propinc}
   Suppose that (A1) holds.
   \begin{itemize}
  \item[(i)]
    There exists $C \in (0,\infty)$ for which, for any $n \in \N$,
     \begin{align}
   \label{incbound}
  \Pr_n [ | \| Y_{n+1} \| - \| Y_n \|  | >  C   ]  =0, \as
 . \end{align}
 In addition
    \begin{align}
   \label{lem3eq}
\limsup_{n \to \infty} \| Y_n \| = \infty, \as.
\end{align}
 \item[(ii)]
    If $\beta \in [0,1)$ then, a.s.,
      \begin{align}
   \label{lem2eq}
  \Exp_n [ \| Y_{n+1} \| - \| Y_n \|  ] =
         \rho \| Y_n \|^{-\beta}
   - \frac{\| Y_n \| }{n+1}   + O ( \| Y_n \|^{-\beta} (\log \| Y_n \|)^{-2} ) . \end{align}
   \item[(iii)]
Suppose also that  (A2) holds and $\beta \geq 1$.
Then, a.s.,
       \begin{align}
   \label{lem2eq2}
  \Exp_n [ \| Y_{n+1} \| - \| Y_n \|   ] & =
        \left( \rho \1_{\{ \beta = 1 \}}  + \frac{1}{2} (d-1) \sigma^2 \right)
       \| Y_n \|^{-1}
   - \frac{\| Y_n \| }{n+1}  \nonumber\\
   & ~~
      + o( \| Y_n \|^{-1}  ( \log \| Y_n \|)^{-1}
       )  \\
  \label{mom2ex}
   \Exp_n [ (\| Y_{n+1} \| - \| Y_n \| )^2   ] & = \sigma^2 + O( n^{-1} \| Y_n \| )
  + o ( ( \log \| Y_n \|)^{-1} ).\end{align}
       \end{itemize}
     \end{proposition}

     We prove
  Proposition \ref{propinc} via a series of lemmas.
 The first gives information on
 the increments of the
 process given by the distance of a general stochastic process
 to its centre-of-mass. In particular, it shows that
 $\|Y_n\|$ inherits boundedness of jumps from $X_n$, and gives
 an expression for the increments of $\|Y_n\|$ in terms
 of $\Delta_n$, the increments of $X_n$.

  \begin{lm}
  \label{lem1}
  Suppose that $(X_n)_{n \in \N}$ is a stochastic process on $\R^d$, and
  $(G_n)_{n \in \N}$ is its centre-of-mass process
  as defined by (\ref{com}). Suppose that $X_1 \in \R^d$ is fixed and that
  (\ref{bound})
  holds for some $B \in (0,\infty)$.
     There exists $C \in (0,\infty)$ for which, for all $n \in \N$,
     (\ref{incbound}) holds.
  Moreover,  a.s.,
   \begin{align}
   \label{lem1eq}
  \| Y_{n+1} \| - \| Y_n \| =
    \frac{n   }{n+1}
  \left(     \frac{ Y_n \cdot \Delta_n }{\| Y_n \| }
     + \frac{\| \Delta_n \|^2}{2 \| Y_n \|} - \frac{ (Y_n \cdot \Delta_n)^2}{2 \| Y_n \|^3} \right)
     + O ( \|Y_n \|^{-2} )
   - \frac{\|Y_n \| }{n+1}  . \end{align}
  \end{lm}
    \begin{proof}
    We work with the process $(\|Y_n \|)_{n \in \N}$.
    From (\ref{bound}) and the triangle inequality, we have the
    simple bound $\| X_n \| \leq \| X_1 \| + B (n-1)$ a.s., for all $n \in \N$.
    Applying the triangle inequality in (\ref{com}) then yields the
    equally simple bound
    \[ \| G_n \| \leq \frac{1}{n} \sum_{i=1}^n ( \| X_1 \| + B (i-1) )
     \leq \| X_1 \| + \frac{Bn}{2} .\]
     Combining these two inequalities together with the fact that $\| Y_n \| \leq \|X_n \| + \| G_n\|$,
     it follows that $\| Y_n \| \leq 2 \| X_1 \| + (3 Bn/2)$ a.s., for all $n \in \N$.
Then from the triangle inequality and (\ref{yeq}) we have that
  \[ \left| \| Y_{n+1} \| - \| Y_n \| \right| \leq \| Y_{n+1} - Y_n \|
  \leq \frac{1}{n} \| Y_n \| + \| \Delta_n \| \leq \frac{5B}{2} + \frac{2 \| X_1 \|}{n} ,\]
  a.s., by (\ref{bound}), and this lattermost quantity is uniformly bounded. Thus we have
    (\ref{incbound}).

  For the final statement of the lemma, note that, from (\ref{yeq}),
  \begin{align}
  \label{change}
    \| Y_{n+1} \| = \frac{n}{n+1} \left( \| Y_n \|^2 + \| \Delta_n \|^2 + 2 Y_n \cdot \Delta_n  \right)^{1/2} .\end{align}
  Now writing $\by = Y_n$ for convenience, we obtain from (\ref{change}) that
  \begin{align}
  \label{change2}
  \| Y_{n+1} \| - \| Y_n \| = \| \by \| \left[ \frac{n}{n+1}
  \left( 1 + \frac{\| \Delta_n  \|^2 + 2 \by \cdot \Delta_n }{\| \by \|^2 } \right)^{1/2} - 1 \right] .
  \end{align}
  Using Taylor's formula for $(1+x)^{1/2}$ with Lagrange remainder in (\ref{change2}) implies
  that
  \begin{align*}
  \| Y_{n+1} \| - \| Y_n \| =   \frac{n \| \by \| }{n+1}
  \left(   \frac{\| \Delta_n  \|^2 + 2 \by \cdot \Delta_n }{2 \| \by \|^2 }
  -   \frac{ ( \| \Delta_n  \|^2 + 2 \by \cdot \Delta_n )^2 }{8 \| \by \|^4 }  
  + O ( \| \by \|^{-3} )
  \right)  - \frac{\| \by \| }{n+1}  , \end{align*}
  using (\ref{bound}) for the error bound.
 Simplifying and again using (\ref{bound}), this becomes
   \begin{align*}
  \| Y_{n+1} \| - \| Y_n \| =
    \frac{n \| \by \| }{n+1}
  \left(    \frac{\| \Delta_n  \|^2}{2 \| \by \|^2}  + \frac{ \by \cdot \Delta_n }{\| \by \|^2 }
  - \frac{ (   \by \cdot \Delta_n )^2}{2 \| \by \|^4 }    + O ( \| \by \|^{-3} )
  \right)  - \frac{\| \by \| }{n+1}  . \end{align*}
  Then equation (\ref{lem1eq}) follows.
     \end{proof}

  Now we turn   to the model defined by (A1), starting with
   the drift
   of $\| Y_n \|$. For $a, b \in \R$, we use the standard notation $a \wedge b := \min \{ a,b\}$.

   \begin{lm}
   \label{lem2}
    Suppose that (A1) holds.
 Then
  the drift of $\| Y_n \|$ satisfies, a.s.,
     \begin{align}
   \label{lem2eq20}
  \Exp_n [ \| Y_{n+1} \| - \| Y_n \|   ] & =
    \frac{n   }{n+1}
      \left( \rho \| Y_n \|^{-\beta} + \Theta_n \| Y_n \|^{-1} \right)
   - \frac{\| Y_n \| }{n+1}  \nonumber\\
   & ~~~ + O( \| Y_n \|^{-(1 \wedge \beta)} (\log \| Y_n \|)^{-2} ) ,\end{align}
   where $\Theta_n$ is the $\F_n$-measurable random variable given by
     \begin{equation}
\label{rhop}
 \Theta_n =   \frac{1}{2} \| Y_n \|^{-2} \Exp_n [ \| Y_n \|^2 \| \Delta_n \|^2 - (Y_n \cdot
\Delta_n)^2   ] .\end{equation}
Moreover, there exists  $C < \infty$ such that $\Theta_n \in [0,C]$ a.s., and
if $\beta \in [0,1)$, (\ref{lem2eq}) holds.
  \end{lm}
  \begin{proof} Taking expectations in (\ref{lem1eq}), using
the fact that
\[   \| Y_n \|^{-1} \Exp_n [ Y_n \cdot \Delta_n   ] = \rho \| Y_n \|^{-\beta}  + O ( \| Y_n\|^{-\beta} (\log \| Y_n \|)^{-2} ) , \]
by (\ref{drift}), we obtain
\begin{align*}
  \Exp_n [ \| Y_{n+1} \| - \| Y_n \|   ]
=
& \frac{n}{n+1} \left( \rho \| Y_n \|^{-\beta} + \frac{1}{2\| Y_n \|^3} \Exp_n [ \| Y_n \|^2 \| \Delta_n \|^2 - (Y_n \cdot
\Delta_n)^2   ] \right) \\
&
+ O( \| Y_n \|^{-(\beta \wedge 1)} ( \log \| Y_n \| )^{-2} ) - \frac{ \| Y_n \|}{n+1} .\end{align*}
By the fact that $|Y_n \cdot \Delta_n| \leq \| Y_n \| \| \Delta_n\|$
and the jumps bound (\ref{bound})  we have that
\[ 0 \leq \Exp_n [ \| Y_n \|^2 \| \Delta_n \|^2 - (Y_n \cdot
\Delta_n)^2   ] \leq C \| Y_n \|^2, \as, \]
for some $C \in (0,\infty)$.
Thus defining $\Theta_n$ by (\ref{rhop}) we obtain (\ref{lem2eq20}) and the fact that
$\Theta_n \in [0,C]$ a.s..
Then (\ref{lem2eq}) follows when $\beta \in [0,1)$.
   \end{proof}

   The next result shows
  how the ellipticity condition (\ref{ue})
  leads to (\ref{lem3eq}).

 \begin{lm}
 \label{lem3}
  Suppose that (A1) holds.
  Then $\limsup_{n \to \infty} \| Y_n \| = \infty$ a.s..
  \end{lm}
     \begin{proof}
      We have from (\ref{change}) that
  \begin{equation}
  \label{0422a}
   \| Y_{n+1} \|^2 - \| Y_n \|^2 = \left( \frac{n}{n+1} \right)^2 \left( \| \Delta_n \|^2 + 2 Y_n \cdot \Delta_n \right)
   - \frac{2n+1}{(n+1)^2} \| Y_n \|^2 .\end{equation}
Fix $p \in \N$, and define   $F_{n,1} := \cap_{i=np}^{np+(p-1)} \{ \Delta_i \cdot \hat Y_i \geq \eps_0 \}$ and
$F_{n,2} := \{ \| Y_{np} \| \leq \frac{\eps_0 np}{16} \}$. Fix also $n_p \in \N$ with $\eps_0 n_p  \geq 16
C$, where $C$ is as in (\ref{incbound}), and consider  $n \geq n_p$ only.
By   (\ref{ue}) we have that   $\Pr_n [ F_{n,1}] \geq \eps_0^p$ a.s.,  and
 hence   L\'evy's extension of the second Borel--Cantelli lemma (see e.g.\
\cite[Theorem 5.3.2]{durrett}) implies that
$\Pr [F_{n,1}\; {\rm i.o.}]=1$. 

Now, observe from  (\ref{incbound}) that
$\|Y_{i+1}\| \leq \| Y_i \| + C$, a.s., which implies   on $F_{n,2}$ that
$\|Y_i\| \leq  \frac{1}{8}\eps_0 np$  for all $i \in \{ np, \ldots np+(p-1) \}$.
   Then, on $F_{n,1} \cap F_{n,2}$, we obtain  from (\ref{0422a}) that, a.s.,
 \begin{align*}
 \| Y_{i+1} \|^2 - \| Y_i \|^2 \geq -\frac{2}{i} \| Y_i \|^2 + \frac{1}{4} \left( \eps_0^2 + 2 \eps_0 \| Y_i \| \right)
 \geq \frac{1}{4} \eps_0 \| Y_i \| + \frac{1}{4} \eps_0^2 \geq  \frac{1}{4} \eps_0^2,
\end{align*}
 for any $i$ with $np \leq i \leq np+(p-1)$ and any $n \geq n_p$.
  Hence on $F_{n,1} \cap F_{n,2}$, a.s.,
  \begin{align*}
     \| Y_{(n+1)p} \|^2 = \| Y_{np} \|^2
 +\sum_{i=np}^{np+(p-1)}(\| Y_{i+1} \|^2 - \| Y_i \|^2)
    \geq p\eps_0^2/4 
   .\end{align*}
Thus, up to sets of probability zero,
 $\{ ( F_{n,1} \cap F_{n,2} ) \; {\rm i.o.}\} \subseteq \{  \limsup_{n \to \infty} \| Y_n \| \geq p^{1/2}\eps_0/2\}$.
Moreover, by definition of $F_{n,2}$,
 $\{ F_{n,2}^{\rm c} \; {\rm i.o.} \} \subseteq \{ \limsup_{n \to \infty} \| Y_n \| = \infty \}$.
Since $\{F_{n,1} \; {\rm i.o.} \} \subseteq \{  ( F_{n,1} \cap F_{n,2} ) \; {\rm i.o.}\} \cup \{ F_{n,2}^{\rm c} \; {\rm i.o} \}$, 
it follows that $\{ F_{n,1} \; {\rm i.o.} \} \subseteq \{  \limsup_{n \to \infty} \| Y_n \| \geq p^{1/2}\eps_0/2 \}$.
Since $p$ was arbitrary, the result follows
from the fact that $\Pr [F_{n,1}\; {\rm i.o.}]=1$, as shown in the first part of this proof.
\end{proof}

  When (A1) holds with $\beta \geq 1$,
  we need more regularity to obtain a well-behaved
  version of (\ref{lem2eq20}).
  Thus we
  impose (A2) and use the following result,
 which in addition gives an expression for the second moment
 of the increment $\|Y_{n+1} \| - \| Y_n \|$.

   \begin{lm}
  \label{lem4}
    Suppose that (A1) and (A2) hold.
  Then $\Theta_n$ as defined by (\ref{rhop}) satisfies
  \begin{align}
  \label{rhopex}
   \Theta_n =   \frac{1}{2} (d-1) \sigma^2 + o( ( \log \| Y_n \|)^{-1} ), \as.\end{align}
  Moreover,
  (\ref{mom2ex}) holds.
  \end{lm}
  \begin{proof}
   First we prove (\ref{rhopex}).
  We have that
 \[
   \Exp_n [ \| \Delta_n \|^2   ] = \sum_{i=1}^d \Exp_n [ (\Delta_n^{(i)})^2
   ] = d \sigma^2 + o( ( \log \| Y_n \|)^{-1} ) ,\]
  by (\ref{cov1}). Also if $Y_n = (y_1,\ldots,y_d) \in \R^d$,
  with the convention that an empty sum is $0$,
  \begin{align}
  \label{lem4a}
  \Exp_n [ (Y_n \cdot \Delta_n)^2   ]
  & = \sum_{i=1}^d y_i^2 \Exp_n [ (\Delta_n^{(i)})^2   ]
  + 2 \sum_{i=2}^d \sum_{j=1}^i y_i y_j \Exp_n [ \Delta_n^{(i)} \Delta_n^{(j)}   ] \nonumber\\
  & = \| Y_n \|^2  \left[ \sigma^2 + o ( (\log \| Y_n \|)^{-1} ) \right] ,\end{align}
  by (\ref{cov1}) and (\ref{cov2}).
  Then (\ref{rhopex}) follows from (\ref{rhop}).

  Next we prove (\ref{mom2ex}). Squaring both sides of (\ref{lem1eq}) and taking expectations
  we obtain
  \[ \Exp_n [ (\| Y_{n+1} \| - \| Y_n \| )^2   ]
  =   \| Y_n \|^{-2} \Exp_n [ (Y_n \cdot
  \Delta_n)^2   ] + O( n^{-1} \| Y_n \| ) + O( \| Y_n \|^{-1} ).\]
  Now using (\ref{lem4a}) yields (\ref{mom2ex}).
  \end{proof}

  \begin{proof}[Proof of Proposition \ref{propinc}.]
   We  collect results from Lemmas \ref{lem1}, \ref{lem2}, \ref{lem3}, and \ref{lem4}. \end{proof}

  \section{Recurrence classification for processes satisfying equation (\ref{drift0b})}
  \label{1dproc}

  \subsection{Introduction}

  In this section we state general results for processes
  with drift of the form (\ref{drift0}). We will later apply
  these results to the process $\| X_n - G_n\|$ satisfying
  (A1) (and maybe also (A2)), but for this section
  we work in some generality.

  Let $(Z_n)_{n \in \N}$ be a stochastic process taking
  values in an unbounded subset $\SS$ of
  $[0,\infty)$, adapted to a filtration $(\F_n)_{n \in \N}$.
  Recall the definition of $\mu_k (n ; x)$ from (\ref{mudef}), so that
  $\Exp_n [ (Z_{n+1} - Z_n)^k ] = \mu_k ( n ; Z_n )$ a.s..
  As discussed in Section \ref{sec:lampsrw}, the case
  where $\mu_2 (n ; x)$ is $O(1)$ and $\mu_1 (n ; x) \to 0$
  as $x \to \infty$ arises often in applications;
   the case where $\mu_1 (n ; x) \to 0$ uniformly
  in $n$ is sometimes known as Lamperti's problem after
  Lamperti's work \cite{lamp1,lamp2,lamp3}.
  Roughly speaking, the Lamperti problem has
  $\mu_1 (n ; x) \approx \rho x^{-\beta}$, $\beta >0$, $\rho \in \R$,
  ignoring higher-order terms.
  Results of Lamperti \cite{lamp1,lamp3} imply that the case
  $\beta =1$ is critical from the point of view of the recurrence
  classification. The supercritical case, when $\beta \in [0,1)$, $\rho>0$,
  has also been studied (see \cite{superlamp} and references therein).

  In this section
  we study the analogous problem for which
  $\mu_1 (n ; x) \approx \rho x^{-\beta} - (x/n)$. In keeping with the
  applications of the present paper, and to ease technical
  difficulties, we adopt
  some stronger regularity assumptions than imposed
  in \cite{lamp1,lamp3} or \cite{superlamp}. Nevertheless,
  this version of the problem is more difficult than
  the classical case (without the extra $-x/n$ term in the
  drift). Thus although the ideas in this section
  are related to those in \cite{lamp1,lamp3} and \cite{superlamp},
  we have to proceed somewhat differently. In particular,
  to obtain our $\beta < 1$ law of large numbers in this setting
  (an analogue of \cite[Theorem 2.3]{superlamp}
  for the standard Lamperti case), we use a `stochastic
  approximation' result (Lemma \ref{lower-lm}), the proof of which uses 
  ideas somewhat
  similar to those in \cite{mv2,superlamp}.

  We impose some regularity conditions on $(Z_n)_{n \in \N}$.
  Suppose that there exists $C \in (0,\infty)$ such that
for all $n \in \N$,
  \begin{equation}
  \label{z1b}
   \Pr_n [ | Z_{n+1} - Z_n | > C   ]   = 0, \as .
  \end{equation}
  We also assume that
     \begin{equation}
   \label{z4b}
   \limsup_{n \to \infty} Z_n = \infty, \as,
   \end{equation}
without which the question of whether $(Z_n)_{n \in \N}$
is recurrent or transient (in the sense
of Definition \ref{def1}) is trivial.
  Note that (\ref{z4b}) is implied by a suitable `irreducibility' assumption, such as,
  for all $y >0$,
 $\inf_{n \in \N} \Pr_n [ Z_m - Z_n > y \mbox{ for some }m>n  ] > 0$, a.s..
 In our case, as in the standard Lamperti problem,
 we will see a distinction between the `critical'
 case where $\beta=1$ and the `supercritical' case where $\beta \in [0,1)$.
 Thus we deal with these two cases separately in the remainder
 of this section.

  \subsection{The critical case: $\beta = 1$}

For $x>0$ and $n \in \N$ define
\begin{equation}
\label{rdef}
r(n;x) :=   n^{-1} x^2 + (\log (1+x) )^{-1}
  .
  \end{equation}
  For $p>0$ we write $\log^p x$ for $(\log x)^p$.
  We impose the further assumptions that
  there exist $\rho' \in \R$ and $s^2 \in (0,\infty)$ such that
  \begin{align}
  \mu_1 (n ; x) & =   \rho' x^{-1} - \frac{x}{n} + o ( x^{-1} r ( n; x) ) , \label{z2b} \\
  \mu_2 (n ; x) & = s^2 +o( r(n ;x) ) . \label{z3b}
  \end{align}

\begin{theo}
\label{thm1}
Suppose that the process $(Z_n)_{n \in \N}$ satisfies
 (\ref{z1b}), (\ref{z4b}), (\ref{z2b}) and (\ref{z3b}) for some
 $\rho' \in \R$ and $s^2 \in (0,\infty)$.
Then if $2\rho' \leq s^2$, $Z_n$ is recurrent.
\end{theo}
\begin{proof}
 Let $W_n:=\log \log Z_n$. Write $D_n := Z_{n+1} - Z_n$.
 First note that Taylor's formula implies that for $x, h$ with
 $x \to \infty$ and
 $h = o(x / \log x)$,
 \[ \log \log ( x+h) = \log \log x + \frac{h}{x \log x}
 - \frac{( \log x + 1) h^2}{2 x^2 \log^2 x } + O ( h^3 x^{-3} (\log x)^{-1} ) .\]
Setting $x = Z_n$ and $h= D_n$ and then taking expectations, we obtain
 \begin{align*}
 \Exp_n [ W_{n+1} - W_n   ]  
  = \frac{\mu_1 (n ; Z_n)}{Z_n \log Z_n} - \frac{ (\log Z_n +1)\mu_2 (n ; Z_n)}{2 Z_n^2 \log^2 Z_n } + O(Z_n^{-3}) ,
 \end{align*}
 using (\ref{z1b}) for the error term.
By (\ref{z2b}) and (\ref{z3b}) this last expression is
\[ \frac{2 \rho' - s^2}{2 Z_n^2 \log Z_n} - \frac{s^2}{2 Z_n^2 \log^2 Z_n} - \frac{1}{n \log Z_n}
  + o ( Z_n^{-2} (\log Z_n)^{-1} r (n ;Z_n )) < 0 , \]
 for all $n$ and $Z_n$ large enough, provided $2 \rho' - s^2 \leq 0$,
 by (\ref{rdef}). Thus there exist non-random constants $w_0 \in (0,\infty)$
 and $n_1 \in \N$ for which,  for all $n \geq n_1$, on $\{ W_n > w_0 \}$,
 \[ \Exp_n [ W_{n+1} - W_n ] < 0 , \as.\]

By Doob's decomposition, we may write $W_n = M_n + A_n$, $n \geq n_1$, where
$W_{n_1} = M_{n_1}$, $(M_n)_{n \geq n_1}$ is a martingale, and the previsible
sequence $(A_n)_{n \geq n_1}$ is defined by
\[ A_n = \sum_{m=n_1}^{n-1} \Exp_m [ W_{m+1} - W_m ]
\leq \sum_{m=n_1}^{n-1} \Exp_m [ W_{m+1} - W_m ] \1 \{ W_m \leq w_0 \}
\leq  C  \sum_{m=n_1}^{n-1}   \1 \{ W_m \leq w_0 \},\]
since the uniform jumps bound (\ref{z1b}) for $Z_n$ implies a uniform jumps bound
for $W_n$, $n \geq n_1$. Hence
$W_n \to \infty$ implies that  $\limsup_{n \to \infty}  A_n  < \infty$ so $M_n \to \infty$ also. However,
$(M_n)_{n \geq n_1}$ is a martingale with uniformly bounded increments (by (\ref{z1b}))
so $\Pr  [ M_n \to \infty] =0$ (see e.g.\ \cite[Theorem 5.3.1, p.\ 204]{durrett}).
Hence $\Pr [ \liminf_{n \to \infty} W_n < \infty ] = 1$. \end{proof}

\begin{theo}
\label{thm2}
Suppose that the process $(Z_n)_{n \in \N}$ satisfies
 (\ref{z1b}), (\ref{z4b}), (\ref{z2b}) and (\ref{z3b})
 for some $\rho' \in \R$ and $s^2 \in (0,\infty)$.
  Then if $2 \rho' > s^2$, $Z_n$ is
transient.
\end{theo}
\begin{proof}
This time set
$$
W_n: =\frac {1}{ \log Z_n}+\frac {9}{\log n}.
$$
Again write $D_n := Z_{n+1} - Z_n$.
We want to compute
\begin{align}
\label{lem8a}
 \Exp_n [ W_{n+1} - W_n  ] = \Exp_n \left[ (\log (Z_n + D_n))^{-1} - (\log Z_n)^{-1}  \right] \nonumber\\
+ 9 \left[ (\log (n+1))^{-1} - (\log n)^{-1} \right] .\end{align}
Observe that, for the final term on the right-hand side of (\ref{lem8a}),
\begin{equation}
\label{lem8b}
 (\log (n+1))^{-1} - (\log n)^{-1} = \frac{ \log ( 1 - (n+1)^{-1} )}{ \log n \log (n+1)}
= - \frac{1}{n \log^2 n } + O(n^{-2}).\end{equation}
Also for the expectation on the right-hand side of (\ref{lem8a}) we have that
\begin{align*}
\Exp_n \left[ (\log (Z_n + D_n))^{-1} - (\log Z_n)^{-1}   \right]  \\
= (\log Z_n)^{-1} \Exp_n \left[ \left( 1 + \frac{ \log ( 1 + (D_n/Z_n))}{\log Z_n} \right)^{-1} -1   \right] .
\end{align*}
Taylor's formula implies that for $a = O(1)$ and $y = o(1)$,
\[ \left( 1 + a \log (1+y ) \right)^{-1} = 1 - ay + \frac{2a^2+a}{2} y^2 + O(y^3) .\]
Applying this formula with $a = 1/\log Z_n$ and $y = D_n/Z_n$ we obtain,
\begin{align*}
 & ~~~ \Exp_n \left[ (\log (Z_n + D_n))^{-1} - (\log Z_n)^{-1}  \right] \\
 & = - \frac{\mu_1 (n ; Z_n)}{Z_n \log^2 Z_n} + \frac{\mu_2 (n ; Z_n)}
 {2 Z_n^2 \log^2 Z_n } + \frac{\mu_2 (n ; Z_n)}
 { Z_n^2 \log^3 Z_n } + O(Z_n^{-3} ) , \end{align*}
 by (\ref{z1b}). Now using (\ref{z2b}) and (\ref{z3b}) we obtain,
 \begin{align}
\label{lem8c}
 & ~~~ \Exp_n \left[ (\log (Z_n + D_n))^{-1} - (\log Z_n)^{-1}   \right] \nonumber\\
& =  \frac{1}{2Z_n^2 \log^2 Z_n}
\left( - (2 \rho' - s^2) + o(r(n;Z_n)) + O ( ( \log Z_n)^{-1} ) \right) + \frac{1}{n \log^2 Z_n} .\end{align}
Suppose that $2 \rho' - s^2 \geq 2\eps > 0$.
Then by (\ref{rdef}), (\ref{lem8a}), (\ref{lem8b}) and (\ref{lem8c})
we have that there exist non-random constants $n_0 \in \N$ and $x_0 \in (1,\infty)$ 
such that for all $n \geq n_0$,
 on $\{ Z_n \geq x_0\}$, a.s.,
\begin{equation}
\label{lem8d}
  \Exp_n [ W_{n+1} - W_n  ]
\leq
-\frac{\eps}{2 Z_n^2 \log^2 Z_n} - \frac{8}{n \log^2 n} + \frac{3}{2n \log^2 Z_n} .\end{equation}
We have that the right-hand side of (\ref{lem8d})
is bounded above
by
\[ \frac{1}{\log^2 Z_n} \left( \frac{-\eps}{2 Z_n^2} + \frac{3}{2n} \right)
\leq \frac{1}{\log^2 Z_n} \left( \frac{-\eps}{2 Z_n^2} + \frac{3\eps}{8Z_n^2} \right) ,\]
provided $n \geq 4 Z_n^2 \eps^{-1}$, and this last upper bound is negative for $Z_n \geq x_0$.
On the other hand,
if $n \leq 4Z_n^2 \eps^{-1}$
 the right-hand side of (\ref{lem8d}) is bounded above
by
\begin{align*}
\frac{1}{n} \left( \frac{3/2}{\log^2 Z_n} - \frac{8}{\log^2 n} \right)
\leq \frac{1}{n} \left( \frac{7}{\log^2 n} - \frac{8}{\log^2 n} \right) < 0 ,\end{align*}
for $Z_n \geq x_0$ and $n \geq n_0$.
Thus in either case we have concluded that for all $n \geq n_0$,
on $\{ Z_n \geq x_0 \}$,
\begin{equation}
\label{superm}
 \Exp_n [ W_{n+1} - W_n  ]
< 0, \as.
\end{equation}

Now fix $K>1$ and $x_1 \geq x_0$. Define the stopping times
\[
\sigma_K := \min \{ n \geq \max \{ n_0, x_1^{18K} \} : Z_n \geq x_1^{4K} \} ;
~~~ \tau_K := \min \{ n \geq \sigma_K : Z_n \leq x_1 \} .
\]
By   (\ref{z4b}) we have that $\Pr [ \sigma_K < \infty ] = 1$.
From (\ref{superm}) and the definition of $\tau_K$ we have that
$(W_{n \wedge \tau_K})_{n \geq \sigma_K}$
is a non-negative supermartingale, and hence it converges almost surely
to a $[0,\infty)$-valued random variable $W := W^{(K)}$. In particular, since $\sigma_K<\infty$ a.s.,
we have
$\lim_{n \to \infty} W_{n \wedge \tau_K } = W$, a.s.. Moreover
\begin{equation}
\label{0401a}
 \Exp [ W] \geq \Exp [ W \1_{\{ \tau_K < \infty \}} ]
= \Exp [ W_{\tau_K}  \1_{\{ \tau_K < \infty \}} ]
\geq \frac{ \Pr [ \tau_K < \infty ] }{ \log x_1 } ,
\end{equation}
since $Z_{\tau_K} \leq x_1$. On the other hand, since $(W_{n \wedge \tau_K})_{n \geq \sigma_K}$
is a   supermartingale,
\begin{equation}
\label{0401b}
 \Exp [ W ] \leq   \Exp [ W_{\sigma_K } ] \leq \frac{1}{4K \log x_1} + \frac{9}{18K \log x_1 }
= \frac{3}{4K \log x_1 },
\end{equation}
using the facts that $Z_{\sigma_K} \geq x_1^{4K}$ and $\sigma_K \geq x_1^{18K}$.
Combining (\ref{0401a}) and (\ref{0401b}) we see that
\[ \frac{ \Pr [ \tau_K < \infty ] }{ \log x_1 }  \leq \frac{3}{4K \log x_1 } .\]
On $\{ \sigma_K < \infty \} \cap \{ \tau_K = \infty \}$, we have that $\liminf_{n \to \infty} Z_n \geq x_1$, so the preceding argument shows
that $\Pr [ \liminf_{n \to \infty} Z_n \geq x_1 ] \geq 1 - \frac{3}{4K}$ for any $K$ and any $x_1 \geq x_0$.
It follows that $\Pr [ Z_n \to \infty ] =1$. \end{proof}

  \subsection{The supercritical case: $\beta \in [0,1)$}

  Once again we will assume that (\ref{z1b}) and (\ref{z4b}) hold. We will also
  assume that there exist $\beta \in [0,1)$ and $\rho \in \R \setminus \{ 0\}$
   such that
  \begin{align}
   \mu_1 (n ; x) & =   \rho x^{-\beta} - \frac{x}{n} + o(x^{-\beta}  ) + o(x n^{-1}) . \label{z2}
   \end{align}

  \begin{theo}
  \label{superc}
  Consider the process $(Z_n)_{n \in \N}$
  satisfying (\ref{z1b}), (\ref{z4b}), and (\ref{z2}),
  where
  $\beta \in [0,1)$. Then $Z_n$ is transient if $\rho>0$ and recurrent
  if $\rho <0$.
  \end{theo}
  \begin{proof}
  First suppose that $\rho>0$. By (\ref{z1b}) we can choose $\rho' \in (0,\infty)$
  so that $2 \rho' > C^2 > \Exp_n [ (Z_{n+1} -Z_n)^2 ]$, a.s., and, by (\ref{z2}),
  \[
  \Exp_n [ Z_{n+1} - Z_n  ]   \geq   ( \rho' +o(1)) Z_n^{-1} - \frac{Z_n}{n} + o ( Z_n^{-1} r (n ;Z_n) ) , \as.
  \]
   It is this inequality, rather than the equality (\ref{z2b}),
  that is needed in the proof of Theorem \ref{thm2}. Hence following that proof implies transience.
  Similarly, if $\rho <0$ we have, for any $\rho' \in (-\infty,0)$, a.s.,
   \[
   \Exp_n [ Z_{n+1} - Z_n   ]   \leq    (\rho'+o(1)) Z_n^{-1} - \frac{Z_n}{n} + o ( Z_n^{-1} r ( n ;Z_n) ) .
   \]
     Using this inequality
   in the proof of Theorem \ref{thm1} implies recurrence. \end{proof}

The rest of this section works towards a proof
 of the following law of large numbers.

       \begin{theo}
   \label{lln1}
  Consider the process $(Z_n)_{n \in \N}$
  satisfying (\ref{z1b}), (\ref{z4b}), and (\ref{z2}),
  where
  $\beta \in [0,1)$ and $\rho>0$. Then, with $\ell (\rho,\beta)$ as defined at (\ref{elldef}),
  as $n \to \infty$,
   \begin{equation}  \label{zlln}
    \frac{Z_n}{n^{1/(1+\beta)}} \toas \ell (\rho, \beta) .
   \end{equation}
  \end{theo}

 The proof uses the following lemma, which is of
 some independent interest, and falls
 loosely into a family of
 ``stochastic approximation'' results; see e.g.~\cite[Section 2.4]{pemantle}.

  \begin{lm}
    \label{lower-lm}
   Suppose that $(V_n)_{n \in \N}$ is a non-negative process adapted to the
   filtration $(\F_n)_{n \in \N}$. Suppose that there exists $r>0$ such that the following hold.
   \begin{itemize}
   \item[(a)]
   There exists  a non-negative sequence
  $(\gamma_n)_{n\in\N}$ adapted to $(\F_n)_{n \in \N}$ with $\sum_{n \in \N} \gamma_n < \infty$ a.s.\
  such that
    for any $b>0$ and all
    $n \in \N$   we have that, a.s.,
    \begin{align*}
    \Exp_n \left[(V_{n+1}-V_n)^2  \right] \le C(b) \gamma_n ~\textrm{on}~
    {\{ V_n\le b \}},
    \end{align*}
   where $C(b)$ is a constant depending only on $b$.
   \item[(b)]
   There exists $\eps>0$, and for any $\delta\in(0,r)$ there is a sequence of events $A_n=A_n(\delta)$, $n \in \N$, such that $A_n \in  \F_n$, 
    $\Pr[ A_n\text{ i.o.}]=0$, and   a.s.\ for all $n \in \N$,
    \begin{align*}
    \Exp_n [V_{n+1}  ]  \le V_n ~\text{on}~    \{ V_n > r + \delta \}\cap A_n^{\rm c},
    ~\text{ and }~ \Exp_n [V_{n+1} ] & \ge V_n ~\textrm{on}~  \{V_n<r-\delta \}\cap A_n^{\rm c},
    \end{align*}
    and also
  $A_n^{\rm c} \subseteq \{ V_{n+1} > (1-\eps) V_n \}$.
    \end{itemize}
Then a.s.\ $\lim_{n\to\infty} V_n = V_\infty$ exists in $[0,\infty)$.  If,
additionally,
   \begin{itemize}
   \item[(c)]
   there exists a non-negative   sequence $({\tilde\gamma_n})_{n\in\N}$ adapted to
$(\F_n)_{n \in \N}$ with $\sum_{n\in\N} \tilde \gamma_n = + \infty$ a.s.\
   such that
    for any $a, b$ with $0<a<b$ and $r\notin [a,b]$, for all $n$ large
enough,
    on {$\{ V_n \in [a,b] \}$}, a.s.,
    \begin{align*}
    \Exp_n  [V_{n+1}-V_n   ] & \le  -  \tilde C(a,b)\tilde\gamma_n
\text{ if } r<a,\\
    \Exp_n  [V_{n+1}-V_n   ] & \ge     \tilde C(a,b)\tilde\gamma_n
\text{ if } r>b,
   \end{align*}
   where $\tilde C(a,b)>0$ is a constant depending only on $a$ and $b$,
    \end{itemize}
then $V_\infty \in\{0,r\}$.
   \end{lm}
\begin{proof}
We first show that under conditions (a) and (b) of the lemma,
$V_n$ converges a.s.\ to some
finite limit $V_\infty$.
We claim that
\begin{equation}
\label{eq:liminf}
\Pr \left[ \{\liminf_{n \to \infty} V_n\le r\}\cup \{\exists \lim_{n \to
\infty} V_n>r\} \right]=1.
\end{equation}
Indeed, suppose that $\{\liminf_{n \to \infty} V_n\le r\}$ does not hold, so that
 $\liminf_{n \to \infty} V_n> r+\delta$ for some $\delta>0$. For $M \in \N$ let
$$
  \tau^{(M)} :=\inf\{n\ge M:\ V_n\le r+\delta\text{ or $A_n(\delta)$ occurs}\},
$$
and define
$V^{(M)}_n :=V_{n \wedge \tau^{(M)}}$.
Then, for each $M$, by (b), $(V^{(M)}_n)_{n\ge M}$, is a non-negative supermartingale and hence converges a.s..
On the other hand, from (b) and our assumption that $\liminf_{n \to \infty} V_n> r+\delta$ it follows that a.s.\ 
$\tau^{(M)}=\infty$ for {\it some\/} $M$; in this case $V^{(M)}_n\equiv V_n$ for all $n\ge M$ and hence $V_n$ must also converge a.s.,
and the limit must be greater than $r$. This establishes (\ref{eq:liminf}).

By an analogous argument with a
bounded submartingale we also establish
\begin{equation}\label{eq:limsup}
\Pr \left[ \{\limsup_{n \to \infty} V_n\ge r\}\cup \{\exists \lim_{n \to
\infty} V_n<r\} \right] =1.
\end{equation}
Given (\ref{eq:liminf}) and (\ref{eq:limsup}),
to show that $\lim_{n \to \infty} V_n$ exists a.s.\ it suffices to
demonstrate this convergence on the set
\begin{equation*}
E:=\{\limsup_{n \to \infty} V_n\ge r\}\cap\{\liminf_{n \to \infty} V_n\le
r\}.
\end{equation*}
Let us prove that on $E$ in fact $\limsup_{n \to \infty} V_n=r$.
For  $\delta>0$ define
$$
E_\delta:=E\cap \{\limsup V_n > y+\delta\} \text{ where }y=r+2\d.
$$
We will show that $\Pr[E_\delta]=0$ for any $\delta>0$,
 which yields the desired conclusion.

Fix some $\nu_0$ such that $V_{\nu_0}>y+\d$.
Iteratively for $i=0,1,2,\dots$ define
   \begin{align*}
       \tau_i& :=\min\{n>\nu_{i}:  V_n\le y-\delta\},\\
    \kappa_i& :=\min\{n>\tau_{i}: V_n> y-\delta\}, \\
    \nu_{i+1}& :=\min\{n>\tau_{i}: V_n\ge y+ \delta\}.
   \end{align*}
On $E$ we have $V_n \leq r + \delta$ infinitely often, so that $V_n \leq
y-\delta$ infinitely often. Thus our definitions imply that $\tau_i$,
{$\kappa_i$}, and $\nu_i$ are
finite for all $i$  on $E_\delta$.  Next, setting $B_n:=\{V_{n-1}\leq y-\delta,\
V_{n}>y\}$, we have by L\'evy's extension of the second Borel--Cantelli lemma (see e.g.\
\cite[Theorem 5.3.2]{durrett})
  \begin{align*}
  \{\{V_{\kappa_i} > y, \; \kappa_i < \infty \} \text{ i.o.}\} \subseteq \{B_n \text{
i.o.}\}=\left\{\sum_{n \in \N} \Pr_n[B_{n+1}]=\infty\right\},
  \end{align*}
  up to events of probability $0$.
On the other hand,
\begin{align}
\label{cheb1}
   \Pr_n[B_{n+1}]
&=
   \Pr_n[V_{n+1}>y] \1\{V_{n} \leq y-\delta\}
\nonumber\\
&\le     \Pr_n [
|V_{n+1}-V_n| > \delta ] \1\{V_{n} \leq y-\delta\}
\nonumber\\
&\le   \delta^{-2} \Exp_n
[ ( V_{n+1} - V_{n} )^2 ] \1 \{ V_{n} \leq  y-\delta  \} ,
\end{align}
by Chebyshev's inequality,
so that by (\ref{cheb1}) and (a), 
\begin{equation}
\label{cheb2}
 \sum_{n \in \N} \Pr_n[B_{n+1}]
 \le \sum_{n \in \N} \delta^{-2} C(y-\delta)  \gamma_{n}< \infty, \as
 \end{equation}   Thus on $E_\delta$, by (\ref{cheb2})
 and the Borel--Cantelli lemma, 
  $\{V_{\kappa_i} > y\}$ occurs only finitely often a.s., 
  so there is some $N_1 \in \N$ for which
$V_{\kappa_i} \in (y-\d, y]$ for all $i \geq N_1$.
Now let
   \begin{align*}
    \eta_i&:=\min\{n>{\kappa_{i}}: V_n \le y-\d\text{ or } V_n \ge y+\d\}
    \le {\nu_{i+1}}.
   \end{align*}
On $E_\delta$ all the $\eta_i$ are also finite (since the $\nu_i$ are
finite).
For $n \in \N$ define
   \begin{align*}
     I_n&=\left\{\begin{array}{ll}
     1, &\text{if } {\kappa_i}\le n<\eta_{i}\text{ for some }i \text{ and } A_n^{\rm c} \text{ occurs};\\
     0, &\text{otherwise},
     \end{array}\right.
     \\
  D_n&=\Exp_n[(V_{n+1}-V_n) I_n] \text{ and }
  M_n=\sum^{n-1}_{s={\kappa_0}} [(V_{s+1}-V_{s})I_s- D_s],
   \end{align*}
   with an empty sum understood as zero so that $M_{n} =0$ for $n \leq \kappa_0$.
Then $(M_n)_{n\in \N}$ is a zero-mean martingale adapted to $(\F_n)_{n\in\N}$, and it is not
hard to see that
$$
  \Exp_n [M_{n+1}^2-M_n^2]=\Exp_n [(M_{n+1}-M_n)^2]\le
\Exp_n[(V_{n+1}-V_n)^2 I_n ].
$$
Moreover, since for ${\kappa_i}\le n< \eta_i$ we have $y-\d<V_n\le
y+\d$, from (a) it follows that
\begin{align*}
  \Exp_n [ M_{n+1}^2- M_n^2] &\le C(y+\d) \g_n.
\end{align*}
This implies that the increasing process  associated with $M_n$ is
bounded by a constant times $\sum_{n\in\N} \gamma_n$ and hence is a.s.\ finite by (a).
Consequently, by  \cite[Theorem 5.4.9]{durrett}
the martingale $M_n$ converges a.s.~to some finite limit; in
particular, there is some $N_2 \in \N$ for which  
$\sup_{n,m \geq N_2} |M_n-M_m|<\d$ a.s.. 
Then for all $i \geq N_1$ such that  $\kappa_i\geq N_2$  we have 
\begin{align*}
   V_{\eta_i}=V_{\kappa_i}+
[M_{\eta_i}-M_{\kappa_i}]+\sum_{s=\kappa_i}^{\eta_i-1}D_s<y+\d ,
\end{align*}
since, by (b), $D_n\le 0$ for $n\in [\kappa_i,\eta_i)$, $n \geq N_2$, and
$V_{\kappa_i}\le y$ 
for all $i \geq N_1$. Consequently, the process $V_n$ eventually
exits the interval $(y-\d,y+\d)$ only on the left (and it cannot jump
over it, as we showed above), contradicting $E_\delta$. So
$\Pr[E_\delta]=0$.

A similar argument shows that on $E$ not only $\limsup_{n \to
\infty}
V_n=r$ but also $\liminf_{n \to \infty} V_n=r$; we sketch the changes needed
to adapt the previous argument to this case. Analogously to $E_\delta$ above, we define
$E'_\delta := E \cap \{ \liminf V_n < y-\delta\}$ where $y = r-2\delta$ and $\delta \in (0,r/3)$.
Also fix some $\nu'_0$ such that $V_{\nu'_0} < y -\delta$, and iteratively set 
 \begin{align*}
       \tau'_i& :=\min\{n>\nu'_{i}:  V_n\ge y+\delta\},\\
    \kappa'_i& :=\min\{n>\tau'_{i}: V_n< y+\delta\}, \\
    \nu'_{i+1}& :=\min\{n>\tau'_{i}: V_n\le y- \delta\}.
   \end{align*}
This time let $B'_n := \{ V_{n-1} \geq y+\delta, V_n < y \}$. Now
by definition of $A_n^{\rm c}$,
$\{ V_n > (1-\eps)^{-1} r \} \cap A_n^{\rm c} \subseteq \{ V_{n+1} > r \} \subseteq (B_{n+1}')^{\rm c}$, so that
\begin{align*} \Pr_n [ B'_{n+1} ] & \leq \Pr_n [ B'_{n+1} ] \1 ( A_n^{\rm c} ) + \1 ( A_n ) \\
& \leq \Pr_n [ B'_{n+1} ] \1 \{ V_n < (1-\eps)^{-1} r \} + \1 ( A_n ) .\end{align*}
A similar argument as that for (\ref{cheb1}) and (\ref{cheb2}), using Chebyshev's inequality and (a), with
$C(y-\delta)$ in (\ref{cheb2}) now being replaced by $C((1-\eps)^{-1} r)$,
shows that, a.s., $\Pr_n [ B'_{n+1} ] \1 \{ V_n < (1-\eps)^{-1} r \}$
is summable, while $\1 ( A_n )$ is a.s.\ summable by assumption in (b). As before, it follows that
$\{ V_{\kappa'_i} < y \}$ a.s.\ occurs only finitely often. Then a similar argument to the previous case, with the martingale $M_n$,
shows that $\Pr[ E'_\delta ] =0$ as well.

Consequently  , on $E$, $\lim_{n
\to \infty} V_n$ a.s.\
exists and equals $r$ in this case. Thus the first claim of the lemma
follows, and
$V_\infty = \lim_{n \to \infty} V_n$ a.s.\ exists in $[0,\infty)$.

  To prove the second claim of the lemma, under the additional condition
(c), we show
  that $\Pr[ V_\infty \in (0,r ) \cup (r, \infty ) ] =0$.
To this end, suppose that $V_n\to y>r$ (the case {$y\in (0,r)$} can be
handled similarly). Choose a small $\d>0$ such that $y-\d>r$.
Then a.s.\ there exists an $N_3$ such that $|V_n-y|<\d$ for all
$n\ge N_3$. Now  define $D'_n=\Exp_n[(V_{n+1}-V_n) ]$
 and
 the martingale
$M'_n=\sum_{s=1}^{n-1} [(V_{s+1}-V_{s})- D'_s]$. 
Then by (c) we have that for all $n \geq N_3$,
$D'_n=\Exp_n[(V_{n+1}-V_n) ]\le -\tilde C(y-\delta, y+\d)\tilde\g_n$.
By a similar argument
to that for $M_n$ above, $M'_n$ must
converge a.s.. However this leads to a contradiction with the
inequality
  \begin{align*}
  [V_{n} - V_{N_3}]-[M'_{n} - M'_{N_3}]&=\sum_{s=N_3}^{n-1}
   D'_s \le - \tilde
   C(y-\d,y+\d)\sum_{s=N_3}^{n-1} \tilde\g_s,
  \end{align*}
since, a.s., as $n\to\infty$ the right-hand side converges to $-\infty$
while the
left-hand side
converges to a finite limit.
  \end{proof}

 Now we can give the proof of Theorem \ref{lln1}.

 \begin{proof}[Proof of Theorem \ref{lln1}.]
It suffices to prove that
  \begin{equation}  \label{zlln2}
     \lim_{n\to\infty}\frac{n}{Z_n^{1+\beta}}=\frac{2+\beta}{\rho(1+\beta)}, \as.
   \end{equation}
Set $V_n=(n-1)/Z_n^{1+\beta}$ and $\tilde
V_n=n/Z_n^{1+\beta}=\frac{n}{n-1}V_n$. Writing  $D_n :=Z_{n+1}-Z_n$, we have
 \begin{align*}
V_{n+1} - V_n &= \tilde V_n \left[ \left(1+ \frac{D_n}{Z_n} \right)^{-(1+\beta)}-
 \left(1-\frac 1n\right)\right]
 =\tilde V_n \left[ \frac 1n -\frac{(1+\beta)D_n}{Z_n} +O( Z_n^{-2}) \right],
 \end{align*}
 using Taylor's formula and (\ref{z1b}) for the error term. Hence
\begin{equation}
\label{eq33}
V_{n+1} - V_n  = \frac{\tilde V_n}n \left[1 -\frac{(1+\beta)n D_n} {Z_n} +O(n Z_n^{-2})\right]
 .\end{equation}
Taking conditional expectations in (\ref{eq33}) we obtain, on $\{ Z_n \to \infty \}$, a.s., 
  \begin{align}\label{eq:edy}
   \Exp_n [ V_{n+1} - V_n ] & =
   \frac{\tilde V_n}n \left[1 -\frac{(1+\beta)n \mu_1 (n ; Z_n) } {Z_n} +O(n Z_n^{-2})\right] \nonumber\\
   &=
   \frac {\tilde V_n}{n}  \left[  2+\beta+o(1) - \left((1+\beta) \rho +o(1)\right) V_n
   \right],
  \end{align}
  using (\ref{z2}), and then using  the fact that $Z_n \to \infty$  to simplify the error terms.
Similarly, squaring both sides of (\ref{eq33}) and taking expectations,
on $\{ Z_n \to \infty \}$, a.s.,
  \begin{align*}
   \Exp_n [ (V_{n+1} - V_n)^2  ]
   & = \frac{\tilde V_n^2}{n^2} \left[ 1 - \frac{2(1+\beta) n \mu_1 ( n ; Z_n)}{Z_n} (1 + o(1)) + \frac{(1+\beta)^2 n^2 \mu_2 (n ; Z_n )}{Z_n^2} (1+o(1)) \right],  
   \end{align*}
   using (\ref{z1b}) to obtain the error terms. Then from  (\ref{z3b}) and (\ref{z2}) we obtain
   \begin{align}
   \label{eq:var}
      \Exp_n [ (V_{n+1} - V_n)^2  ]
   & = \frac{\tilde V_n^2}{n^{2}}
   \left[ 3 + 2\beta + o(1) - \left( 2 \rho (1+\beta) + o(1) \right) V_n + \frac{(c + o(1)) n^2}{ Z_n^{2} }  \right]
     ,\end{align}
   for some $c \in (0,\infty)$ (depending on $s^2$ and $\beta$)
   as $Z_n \to \infty$ and $n \to \infty$. 
   For a fixed $b>0$ and $A < \infty$, there exists a (non-random) $n_0$ for which $\{ V_n \leq b \}$ implies that
   $\{ Z_n \geq A \}$ for all $n \geq n_0$. In particular, from (\ref{eq:var}) we have that
   for  some (non-random) $C(b) < \infty$, on $\{ V_n \leq b\}$,    for any $n \in \N$, a.s.,
   \[  \Exp_n [ (V_{n+1} - V_n)^2  ] \leq \frac{ \tilde V_n^2}{n^2} \left[ O(1) + (c+o(1)) n^2  (\tilde V_n/n)^{\frac{2}{1+\beta}}    \right]
    \leq C(b) n^{-\frac{2}{1+\beta}} .\] 

  Since $\beta <1$, $\sum_{n\in\N} n^{-2/(1+\beta)} < \infty$
  so that the conditions of part (a) of Lemma~\ref{lower-lm}
  are satisfied with the present choice of $V_n$ and $\gamma_n = n^{-2/(1+\beta)}$.
  Let $A_n := \{ Z_n < A \}$. By Theorem \ref{superc}, $Z_n \to \infty$ a.s.,
  so that
  $A_n$ occurs only finitely often for any $A \in (0,\infty)$.
  Taking $r=\frac{2+\beta}{\rho(1+\beta)}$,
the conditions on $\Exp_n [ V_{n+1}]$ in part (b) of Lemma \ref{lower-lm}
are shown to hold for any $\delta \in (0,r)$, taking $A = A(\delta)$ sufficiently large, by (\ref{eq:edy}). 
Indeed, from (\ref{eq:edy}), on $\{ V_n > r +\delta \}$ for some $\delta \in (0,r)$,
\[ 
\Exp_n [ V_{n+1} - V_n ] \leq  - \delta (1+\beta) \rho (1+o(1)) n^{-1} \tilde V_n  ,
\]
which is negative on $A_n^{\rm c}$ for our choice of $A = A(\delta)$.
 A similar argument
holds for the other condition on $\Exp_n [ V_{n+1} ]$
 in Lemma \ref{lower-lm}(b). 
The final condition in (b), that $A_n^{\rm c}$ implies that $V_{n+1} > (1-\eps) V_n$
for some $\eps \in (0,1)$, follows from (\ref{eq33}) and the fact that $D_n$ is uniformly bounded
(by (\ref{z1b})), taking $A$ and $n$ sufficiently large in our choice of $A_n$.

The conditions in part (c) of Lemma
\ref{lower-lm} follow from (\ref{eq:edy}) again, with $\tilde \gamma_n = n^{-1}$,
noting that the $o(1)$ terms in (\ref{eq:edy}) are uniformly small on $\{V_n \leq b\}$
for any $n \geq n_0$ (for some non-random $n_0 \in \N$).

Hence we conclude from Lemma \ref{lower-lm} that
$V_n \to V_\infty$ a.s.~where $V_\infty \in \{ 0 , r\}$.
To complete the proof of the theorem
we must show that $\Pr[ V_n\to 0 ]=0$. This, however, follows from
the fact that $\limsup_{n \to \infty} ( n^{-1/(1+\beta)}  Z_n ) <\infty$ a.s.\ due to
\cite[Theorem 2.3]{superlamp}, noting the remark following that theorem.  \end{proof}

  \section{Proofs of main theorems on self-interacting walks}
  \label{proofs}

  \subsection{Recurrence classification: Proofs of Theorems \ref{ythm1}  and \ref{1dthm}}
  \label{recprf}

   We apply the results of Section \ref{1dproc}
   to  $Z_n = \| Y_n \| = \| X_n - G_n \|$.

   \begin{proof}[Proof of Theorem \ref{ythm1}.]
  Suppose that (A1) and (A2) hold, and that $\beta \geq 1$.
 First note that with $Z_n = \| Y_n \|$, (\ref{incbound}) and (\ref{lem3eq})
 imply (\ref{z1b}) and (\ref{z4b}). Now from (\ref{lem2eq2}) we obtain,
 with $r(n;x)$ defined by (\ref{rdef}),
    \[ \Exp_n [ Z_{n+1} - Z_n  ] = \left( \rho\1_{\{ \beta =1\}} + \frac{1}{2} (d-1) \sigma^2 \right)
  \frac{1}{Z_n} - \frac{Z_n}{n} + o(Z_n^{-1} r(n; Z_n)), \as \]
 Similarly, we have from
  (\ref{mom2ex}) that
  \[
 \Exp_n [ ( Z_{n+1} - Z_n )^2  ] = \sigma^2 +  O( Z_n n^{-1} ) + o ( (\log Z_n)^{-1} ).\]
 First suppose that
 $\beta =1$. Thus (\ref{z2b}) and (\ref{z3b}) hold with $\rho' = \rho + (d-1)(\sigma^2/2)$ and
 $s^2 = \sigma^2$. It follows from Theorems
 \ref{thm1} and \ref{thm2} that
 $Z_n$ is transient if and only if $2 \rho' > s^2$, or equivalently $2 \rho > \sigma^2 ( 2-d )$,
 i.e., $\rho > \rho_0$. This proves part (i) of the theorem.

 Finally suppose that $\beta >1$. This time
  (\ref{z2b}) and (\ref{z3b}) hold with $\rho' = (d-1)(\sigma^2/2)$ and
 $s^2 = \sigma^2$. It follows from Theorems
 \ref{thm1} and \ref{thm2} that
 $Z_n$ is transient if and only if $2 \rho' > s^2$, or equivalently
$ \sigma^2 ( 2-d ) < 0$,
 i.e., $d>2$. This proves part (ii).
 \end{proof}

\begin{proof}[Proof of Theorem \ref{1dthm}.] Suppose that $d=1$. If $Y_n$ is transient,
 then by (\ref{incbound}) we have that with probability 1 either:
(i) $Y_n \to + \infty$; or (ii) $Y_n \to - \infty$. In case (i), there exists $N \in [2,\infty)$
 for which $Y_n \geq 1$ for all $n \geq N$, so  (\ref{gandy}) with (\ref{incbound}) implies
 that for $n\geq N$,
 \[ G_n \geq X_1 - CN + \sum_{j=N}^n \frac{1}{j-1} \to \infty, \as ,\]
 as $n \to \infty$; a similar argument applies in case (ii).
 Since $X_n = Y_n + G_n$,
 and $Y_n$, $G_n$ are transient with the same sign,
 it follows that $X_n$ is transient too. \end{proof}

  \subsection{Limiting directions: Proof of Theorem \ref{dirthm}}
  \label{direction}

  The key first step in the proof of Theorem \ref{dirthm} is the
  following application of the law of large numbers, Theorem \ref{lln1}.

  \begin{lm}
  \label{ylln}
   Suppose that
   (A1) holds with $d \in \N$, $\beta \in [0,1)$, and $\rho > 0$.
As $n \to \infty$,
 \[ n^{-1/(1+\beta)}  \| X_n - G_n \| \toas  \ell (\rho, \beta ).
  \]
  \end{lm}
  \begin{proof} We take $Z_n = \| Y_n \| = \| X_n -G_n\|$ and apply Theorem \ref{lln1}.
The conditions of the latter are verified since (\ref{incbound}), (\ref{lem3eq}),
  and (\ref{lem2eq}) imply
 (\ref{z1b}), (\ref{z4b}), and (\ref{z2}) respectively. \end{proof}

  The second step in the proof of Theorem \ref{dirthm}
  is to show that the process $Y_n = X_n - G_n$ has
  a limiting direction. Together with Lemma \ref{ylln}
  and the simple but useful relation (\ref{gandy}), we will then be
  able to deduce the asymptotic behaviour of $X_n$ and $G_n$.

  We use the notation $\hat Y_n := Y_n / \| Y_n\|$, with the convention that
 $\hat \0 := \0$.
  Then $(\hat Y_n)_{n \in \N}$ is an $(\F_n)$-adapted process,
  and using the vector-valued version of Doob's decomposition we may write
  \begin{equation}
  \label{ddecomp}
  \hat Y_n = A_n + M_n,
  \end{equation} where $M_1 = \hat Y_1$, $(M_n)_{n \in \N}$ is an $(\F_n)$-adapted
  $d$-dimensional martingale and $(A_n)_{n \in \N}$ is the previsible sequence defined by
  $A_1 = \0$ and
  $A_{n} = \sum_{m=1}^{n-1} \Exp_m [ \hat Y_{m+1} - \hat Y_m ]$ for $n \geq 2$.

  \begin{lm}
  \label{dlem1}
   Suppose that
   (A1) holds with $d \in \N$, $\beta \in [0,1)$, and $\rho > 0$.
   Let the Doob decomposition of $\hat Y_n$ be as given at (\ref{ddecomp}).
  There exists a $d$-dimensional random vector $A_\infty$ such that
  $A_n \to A_\infty$ a.s., as $n \to \infty$.
  \end{lm}
  \begin{proof}
  We have from (\ref{yeq}) that, with $\Delta_n := X_{n+1} - X_n$ as usual,
  \begin{align*}
  A_{n+1} - A_n & = \Exp_n \left[ \frac{Y_n + \Delta_n}{\| Y_n + \Delta_n \|} - \hat Y_n \right]
  = \Exp_n \left[ \frac{\Delta_n}{\| Y_n + \Delta_n \|} \right] - \hat Y_n \Exp_n \left[ \frac{ \| Y_n + \Delta_n \| - \| Y_n \|}{  \| Y_n + \Delta_n \|} \right] \\
  & =: T_1 - \hat Y_n T_2.\end{align*}
  We deal with the expectations $T_1$ and $T_2$ separately. First,
  \[  T_1 = \| Y_n \|^{-1} \Exp_n [ \Delta_n ]
  - \Exp_n \left[  \frac{ (\| Y_n + \Delta_n \|  - \| Y_n \|) \Delta_n }{\| Y_n \| \| Y_n + \Delta_n \|  }  \right] .\]
  The numerator in the last expectation is bounded in absolute value by $\| \Delta_n \|^2$, by the triangle inequality.
  Then using the fact that $\| \Delta_n \|$ is uniformly bounded, and that $\| Y_n\| \sim \ell(\rho,\beta) n^{1/(1+\beta)}$ by Lemma
  \ref{ylln}, it follows that
  \[  T_1 = \| Y_n \|^{-1} \Exp_n [ \Delta_n ] + O ( n^{-2/(1+\beta)} ) , \as, \]
  as $n \to \infty$.  Similarly, we have that
  \[ T_2
  = \Exp_n \left[ \frac{ \| Y_n + \Delta_n \|^2 - \| Y_n \|^2}{  \| Y_n + \Delta_n \| (\| Y_n + \Delta_n \| + \| Y_n \| )} \right] .\]
  Again using the boundedness of $\| \Delta_n \|$  and  that $\|Y_n\| \sim \ell(\rho,\beta) n^{1/(1+\beta)}$, we obtain
  \[ T_2 =
  \Exp_n \left[ \frac{ 2 \Delta_n \cdot Y_n + \| \Delta_n \|^2}{  \| Y_n + \Delta_n \| (\| Y_n + \Delta_n \| + \| Y_n \| )} \right] =
   \Exp_n \left[ \frac{\hat Y_n \cdot \Delta_n }{ \| Y_n \|  } \right] + O ( n^{-2/(1+\beta)} ) , \as.\]
  On applying (\ref{drift}) to evaluate the terms $\Exp_n [ \Delta_n]$ and $\Exp_n [ \Delta_n \cdot \hat Y_n ]$,
  the leading terms in $T_1$ and $\hat Y_n T_2$ cancel to give
  \[ A_{n+1} - A_n = O ( \| Y_n \|^{-\beta-1} (\log \| Y_n \| )^{-2} ) +  O ( n^{-2/(1+\beta)} ) .\]
 Since $\|Y_n\| \sim \ell(\rho,\beta) n^{1/(1+\beta)}$, and $\beta <1$,   these two $O (\, \cdot\, )$
  terms are summable, so that $\sum_{n=1}^\infty \| A_{n+1} - A_n \| < \infty$, a.s.,
  implying that $A_n$ converges a.s.. \end{proof}

  \begin{lm}
  \label{dlem2}
     Suppose that
   (A1) holds with $d \in \N$, $\beta \in [0,1)$, and $\rho > 0$.
 Let the Doob decomposition of $\hat Y_n$ be as given at (\ref{ddecomp}).
  There exists a $d$-dimensional random vector $M_\infty$ such that
  $M_n \to M_\infty$ a.s., as $n \to \infty$.
  \end{lm}
  \begin{proof}
  Taking expectations in the vector identity $\|M_{n+1} - M_n\|^2 =\| M_{n+1}\|^2 - \|M_n\|^2 - 2 M_n \cdot (M_{n+1} - M_n)$
  and using the martingale property, we have
  \[ \Exp_n [ \| M_{n+1} \|^2 - \|M_n\|^2 ] = \Exp_n [ \| M_{n+1} - M_n \|^2 ] = \Exp_n [ \| \hat Y_{n+1} - \hat Y_n - \Exp_n [ \hat Y_{n+1} - \hat Y_n ] \|^2 ] .\]
  Expanding out the expression in the latter expectation, it follows that
  \[  \Exp_n [ \| M_{n+1}\|^2 -\| M_n\|^2 ] = \Exp_n [ \| \hat Y_{n+1} - \hat Y_n \|^2 ] - ( \Exp_n [ \hat Y_{n+1} - \hat Y_n ] )^2  \leq \Exp_n [ \| \hat Y_{n+1} - \hat Y_n \|^2 ] .\]
Here we have from (\ref{yeq}) that
\[ \| \hat Y_{n+1} - \hat Y_n \| = \left\| \frac{ Y_n (\| Y_n \| - \| Y_n + \Delta_n \| ) + \Delta_n \| Y_n \|}{\| Y_n \| \|Y_n + \Delta_n \|} \right\|
\leq \frac{ 2 \| \Delta_n \|}{\| Y_n + \Delta_n \|} ,\]
by the triangle inequality.
Since $\| \Delta_n \|$ is uniformly bounded, and $\|Y_n\| \sim \ell (\rho,\beta) n^{1/(1+\beta)}$ by Lemma
  \ref{ylln}, it follows that $\| \hat Y_{n+1} - \hat Y_n \| = O ( n^{-1/(1+\beta)})$,
  so that
\[ \Exp_n [ \| M_{n+1} \|^2 - \|M_n\|^2 ] = O ( n^{-2/(1+\beta)}), \as. \]
Hence $\sum_{n=1}^\infty \Exp_n [ \| M_{n+1} \|^2 - \|M_n\|^2 ] < \infty$, a.s.,
which implies that $M_n$ has an almost-sure limit,
by e.g.\ the $d$-dimensional version of \cite[Theorem 5.4.9, p.\ 217]{durrett}. \end{proof}

\begin{proof}[Proof of Theorem \ref{dirthm}.]
Combining Lemmas \ref{dlem1} and \ref{dlem2} with the decomposition (\ref{ddecomp}), we conclude that $\hat Y_n \to A_\infty + M_\infty =: \bu$,
for some random unit vector $\bu$,
a.s., as $n \to \infty$. In other words, the process $Y_n$ has a limiting direction.
It follows from the representation
(\ref{gandy})
that the processes $G_n$ and $X_n$ have the same limiting direction. Specifically,
\[ G_n = X_1 + \sum_{j=2}^n \frac{1}{j-1} \| Y_j \| \hat Y_j = X_1
+ \sum_{j=2}^n \frac{1}{j-1} [\ell(\rho,\beta) + o(1) ] j^{1/(1+\beta)} [\bu + o(1)], \as, \]
by Lemma \ref{ylln}. Hence
\[ G_n = [ (1+\beta) \ell(\rho,\beta) \bu   + o(1) ] n^{1/(1+\beta)}, \as, \]
and  the result for $X_n$ follows since $X_n = G_n + Y_n$.
  \end{proof}

 \subsection{Upper bounds: Proof of Theorem \ref{extent}}
 \label{boundprf}

 Theorem \ref{extent} will follow
 from the next result, which gives  bounds   for $\| Y_n\|$.

 \begin{proposition}
 \label{extent2}
 Suppose that (A1) holds with $d \in \N$, $\beta \geq 0$, and $\rho \in \R$.
 Then the following bounds apply.
 \begin{itemize}
  \item[(i)]  If $\beta \geq 1$, then for any $\eps>0$,
  a.s., for all but finitely many $n \in \N$,
  $\| Y_n \| \leq   n^{1/2} ( \log n)^{(1/2)+\eps}$.
 \item[(ii)] If (A2) holds, $\beta =1$, and $\rho < - (d\sigma^2/2)$, then for any $\eps>0$,
 a.s., for all but finitely many $n \in \N$, $\| Y_n \| \leq n^{\gamma (d,\sigma^2, \rho) + \eps}$
 where $\gamma (d,\sigma^2, \rho)$ is given by (\ref{gammadef}).
 \item[(iii)]
  If $\beta \in [0,1)$ and $\rho <0$,
  then for any $\eps>0$,
    a.s., for all but finitely many $n \in \N$,
  $\| Y_n \| \leq    (\log n)^{\frac{1}{1-\beta} +\eps}$.
  \end{itemize}
  \end{proposition}

 To prove this result we apply some general results
 from \cite{mvw}.
 Section 4 of \cite{mvw} dealt
 with  stochastic processes that were time-homogeneous,
 but that condition was not used
 in the proofs of the results that we apply here,
 which relied on the very general results
 of Section 3 of \cite{mvw}: the basic tool
 is Theorem 3.2 of \cite{mvw}.

It is most convenient to again work in some generality.
Again let $(Z_n)_{n \in \N}$ denote a stochastic process on $[0,\infty)$.
Recall the definition of $\mu_k(n;x)$ from (\ref{mudef}).
 The next result gives the upper bounds that we need.
 Part (i) is contained in \cite[Theorem
 4.1(i)]{mvw}. Part (ii) is a variation on   \cite[Theorem 4.3(i)]{mvw} that is more suited to the present application.
 Part (iii) is also based on \cite{mvw}
 but does not seem to have appeared before.

 \begin{lm}
 \label{boundslem}
 Suppose that $(Z_n)_{n \in \N}$ is such that
 (\ref{z1b}) holds.
 \begin{itemize}
 \item[(i)] Suppose that for some $A < \infty$,
  $x \mu_1 (n ; x ) \leq A$
 for all $n$ and $x$ sufficiently large. Then for any $\eps>0$,
  a.s., $Z_n \leq   n^{1/2} (\log n)^{(1/2)+\eps}$   for all but finitely many $n \in \N$.
 \item[(ii)] Suppose that for some $v>0$ and $\kappa >1$,
 $2 x \mu_1 (n ; x) \leq - \kappa \mu_2 (n ; x) + o(1)$ and $\mu_2 (n;x) \geq v$
 for all $n$ and $x$ sufficiently large. Then, for any $\eps >0$,
 a.s.~$Z_n \leq   n^{\frac{1}{1+\kappa} +\eps}$   for all but finitely many $n \in \N$.
 \item[(iii)]
   Suppose that
 for some $\beta \in [0,1)$ and $A>0$,
 $x^\beta \mu_1 (x ; n) \leq -A$ for all $n$ and $x$ large enough. Then
 for any $\eps>0$,
 a.s., for all but finitely many $n \in \N$, $Z_n \leq (\log n)^{\frac{1}{1-\beta} +\eps}$.
 \end{itemize}
 \end{lm}
\begin{proof}
First we prove part (ii). Let $\kappa' = \kappa - \eps$ for $\eps \in (0,\kappa)$. Writing $D_n = Z_{n+1} - Z_n$,
\begin{align*}
\Exp_n [ Z_{n+1}^{1+\kappa'} - Z_n^{1+\kappa'} ] & = Z_n^{1+\kappa'} \Exp_n [ (1 + (D_n/Z_n))^{1+\kappa'} -1 ] \\
& = (1+ \kappa') Z_n^{\kappa'} \left( \mu_1 (n ; Z_n) + \frac{\kappa'}{2 Z_n} \mu_2 (n ; Z_n ) + O( Z_n ^{-2} ) \right) ,
\end{align*}
using Taylor's formula and (\ref{z1b}). Under the conditions of part (ii),
we have
\[ \mu_1 (n ; Z_n) + \frac{\kappa'}{2 Z_n} \mu_2 (n ; Z_n ) + O( Z_n ^{-2} ) \leq - \frac{\eps}{2 Z_n} \mu_2 (n ; Z_n) + o (Z_n^{-1})
< 0 ,\]
for all $n$ and $Z_n$ large enough. Hence $\Exp_n [ Z_{n+1}^{1+\kappa'} - Z_n^{1+\kappa'} ]$
is uniformly bounded above and the result follows from Theorem 3.2 of \cite{mvw}.

It remains to prove part (iii).
For $\alpha >0$, define
$f_\alpha (x) := \exp \{ x^\alpha \}$. First we show that, under
the conditions of the lemma,
  for any $\alpha \in (0,1-\beta)$, for some $C< \infty$,
  \begin{equation}
  \label{fbound}
   \Exp_n [ f_\alpha ( Z_{n+1} ) -  f_\alpha ( Z_n ) ] \leq C, \as .\end{equation}
 Writing $D_n = Z_{n+1} -Z_n$, we have that
 \begin{align*}
 \Exp_n [ f_\alpha (Z_{n+1}) -  f_\alpha (Z_n) ]
 = f_\alpha (Z_n) \Exp_n \left[ \exp \{ ( Z_n + D_n)^\alpha - Z_n^\alpha \} - 1 \right] .\end{align*}
 Since $D_n = O(1)$ a.s., by (\ref{z1b}),   Taylor's formula applied to the last expression yields
  \begin{align*}
 \Exp_n [ f_\alpha (Z_{n+1}) -  f_\alpha (Z_n) ]
 = f_\alpha (Z_n) \Exp_n \left[  \alpha D_n Z_n^{\alpha -1} + O(Z_n^{2\alpha -2} )\right] .\end{align*}
 Here we have that $\Exp_n [ D_n ] \leq  - A Z_n^{-\beta}$
 for   all $Z_n ,n$ large enough.
 Since $\alpha < 1-\beta$ we obtain (\ref{fbound}). Now we can
 apply Theorem 3.2 of \cite{mvw} to complete the proof. \end{proof}

Finally  we complete the proofs of Proposition
\ref{extent2} and  Theorem \ref{extent}.

\begin{proof}[Proof of Proposition \ref{extent2}.]
 Under the conditions of part (i)
of Proposition \ref{extent2}, we have from (\ref{incbound})  and Lemma \ref{lem2}
that the conditions of Lemma \ref{boundslem}(i) hold for $Z_n = \| Y_n \|$.
Thus we obtain part (i) of the proposition. Similarly, under the conditions of part (ii), we have from
(\ref{incbound}), (\ref{lem2eq2}) and (\ref{mom2ex})  that   Lemma \ref{boundslem}(ii) holds for $Z_n = \| Y_n \|$
and $\kappa = - \frac{2\rho}{\sigma^2} - (d-1)$, which is greater than $1$ for $\rho < - d \sigma^2/2$.
Finally,  under the conditions of part (iii), we have from
(\ref{incbound}) and (\ref{lem2eq}) that Lemma \ref{boundslem}(iii) holds for $Z_n = \| Y_n \|$.
 \end{proof}

 \begin{proof}[Proof of Theorem \ref{extent}.]
  Part (i) of the theorem follows from Theorem \ref{dirthm}. Parts (ii), (iii), and (iv) follow
 from Proposition \ref{extent2} with (\ref{gandy})
and the triangle inequality; note this introduces an extra logarithmic factor in the case of part (iv) of the theorem.
\end{proof}

 \section*{Acknowledgements}

 FC was partially supported
by CNRS UMR 7599. MM
thanks the Fondation Sciences Math\'ematiques de Paris for support
during part of this work. Some of this work was done while AW
was at the University of Bristol, supported by the Heilbronn Institute
for Mathematical Research.


\begin{thebibliography}{09}

 \bibitem{as} H.C. Akuezue and J. Stringer,
 Random aggregation and random-walking center of mass,
 {\em J. Statist. Phys.} {\bf 56} (1989) 461--470.

 \bibitem{abv} O. Angel, I. Benjamini, and B. Vir\'ag,
 Random walks that avoid their past convex hull,
 {\em Elect. Comm. Probab.} {\bf 8} (2003) 6--16.

\bibitem{ac} R.A. Atkinson and P. Clifford,
The escape probability for integrated Brownian motion
with non-zero drift,
{\em
J. Appl. Probab.} {\bf 31} (1994) 921--929.

\bibitem{blr} M. Bena\"im, M. Ledoux, and O. Raimond,
Self-interacting diffusions,
{\em Probab. Theory Relat. Fields} {\bf 122} (2002) 1--41.

\bibitem{BFV09} V. Beffara, S. Friedli, and Y. Velenik,
 Scaling limit of the prudent
walk,
{\em Elect. Comm. Probab.} {\bf 15} (2010) 44--58.

\bibitem{ck} S. Chambeu and A. Kurtzmann,
Some particular self-interacting diffusions:
 Ergodic behaviour and almost-sure convergence. To appear in {\em Bernoulli}.


 \bibitem{chayes} L. Chayes,
 Ballistic behaviour for biased self-avoiding
 walks, {\em Stochastic Processes Appl.} {\bf 119} (2009) 1470--1478.

\bibitem{chung}
K.L. Chung,
A Course in Probability Theory, 2nd ed.,
Academic Press, San Diego, 1974.

\bibitem{CSY04} F. Comets, T. Shiga, and N. Yoshida, Probabilistic
analysis of directed polymers in a random environment: A review,
in: Stochastic Analysis on Large Scale Interacting Systems, pp.~115--142,
{\em Advanced
Studies in Pure Mathematics} {\bf 39} (Funaki \& Osada, eds.), 2004.

\bibitem{durrett} R. Durrett,
Probability: Theory and Examples, 4th ed., Cambridge University Press,
Cambridge, 2010.

\bibitem{dr} R.T. Durrett and L.C.G. Rogers,
Asymptotic behavior of Brownian polymers,
{\em Probab. Theory Relat. Fields} {\bf 92} (1992) 337--349.

\bibitem{FMM} G. Fayolle, V.A. Malyshev, and M.V. Menshikov,
Topics in the Constructive Theory of Countable Markov Chains,
Cambridge University Press, 1995.

\bibitem{flp}
J.-D. Fouks, E. Lesigne, and M.E. Peign\'e,
\'Etude asymptotique d'une marche al\'eatoire centrifuge,
{\em Ann. Inst. H.
Poincar\'e Probab. Statist.} {\bf  42}  (2006) 147--170.

\bibitem{giac}
G. Giacomin,
Random Polymer Models,
 Imperial College Press, London, 2007.

 \bibitem{grill} K. Grill, On the average of a random
 walk, {\em Statist. Probab. Lett.} {\bf 6} (1988) 357--361.

\bibitem{holl} F. den Hollander,
Random Polymers,
{\em Lecture Notes in Mathematics} {\bf 1974},
Springer-Verlag, Berlin, 2009.

 \bibitem{hughes} B.D. Hughes, Random Walks and Random
 Environments; Volume 1: Random Walks, Clarendon Press,
 Oxford, 1995.

\bibitem{iv} D. Ioffe and Y. Velenik,
Ballistic phase of self-interacting random walks,
in: Analysis and Stochastics of Growth Processes and Interface
Models (M\"orters {\em et al.}, eds.), Oxford University Press, 2008.

\bibitem{iw}
Y. Isozaki and S. Watanabe,
An asymptotic formula for the Kolmogorov diffusion and
a refinement of Sinai's estimates for the integral
of Brownian motion,
{\em Proc. Japan Acad.} {\bf 70} (1994) 271--276.

\bibitem{kenyon}
R. Kenyon,
The asymptotic determinant of the discrete Laplacian,
{\em Acta Math.} {\bf 185} (2000) 239--286.

\bibitem{lamp1} J. Lamperti,
Criteria for the recurrence or
transience of stochastic processes I,
{\em J. Math. Anal. Appl.} {\bf 1} (1960) 314--330.

\bibitem{lamp2} J. Lamperti, A new class
of probability limit theorems,
{\em J. Math. Mech.} {\bf 11} (1962) 749--772.

\bibitem{lamp3} J. Lamperti, Criteria for
stochastic processes II: passage-time moments,
{\em J. Math. Anal. Appl.} {\bf 7} (1963) 127--145.

\bibitem{lawler} G.F. Lawler,
A self-avoiding random walk,
{\em Duke Math. J.} {\bf 47} (1980) 655--693.

\bibitem{lsw2}
G.F. Lawler, O. Schramm, and W. Werner,
Conformal invariance of planar loop-erased random walks and uniform spanning
trees,
{\em Ann. Probab.} {\bf  32}  (2004)  939--995.

\bibitem{lsw} G.F. Lawler, O. Schramm, and W. Werner,
On the scaling limit of planar self-avoiding walk,
in: Fractal Geometry and Applications: A Jubilee of Beno\^it Mandelbrot, Part 2, pp.~339--364,
{\em Proc. Sympos. Pure Math.} {\bf 72}, Amer. Math. Soc., Providence, RI, 2004.

\bibitem{mmw} I.M. MacPhee, M.V. Menshikov, and A.R. Wade,
Angular asymptotics for multi-dimensional non-homogeneous random
walks with asymptotically zero drifts,
{\em Markov Processes Relat. Fields} {\bf 16} (2010) 351--388.

\bibitem{ms} N. Madras and G. Slade,
The Self-Avoiding Walk, Birkh\"auser, 1993.

\bibitem{mvw} M.V. Menshikov, M. Vachkovskaia, and A.R. Wade,
Asymptotic behaviour of randomly reflecting billiards in
unbounded tubular domains,
{\em J. Statist. Phys.}
{\bf 132} (2008) 1097--1133.

\bibitem{mv2} M.V. Menshikov and S. Volkov,
Urn-related random walk with drift $\rho x\sp \alpha/t\sp \beta$,
{\em Electron. J. Probab.} {\bf 13} (2008) 944--960.

\bibitem{superlamp} M.V. Menshikov and A.R. Wade,
Rate of escape and central limit theorem
for the supercritical Lamperti problem,
{\em Stochastic Processes Appl.} {\bf 120} (2010) 2078--2099.

\bibitem{mt} T. Mountford and P. Tarr\`es,
An asymptotic result for Brownian polymers,
{\em Ann. Inst. H. Poincar\'e Probab. Statist.} {\bf 44} (2008) 29--46.

\bibitem{nrw} J.R. Norris, L.C.G. Rogers, and D. Williams,
Self-avoiding random walk: A Brownian model with local time drift,
{\em Probab. Theory Relat. Fields} {\bf 74} (1987) 271--287.

\bibitem{pemantle} R. Pemantle, A survey of random processes with reinforcement,
{\em Probab. Surv.} {\bf 4} (2007) 1--79.

\bibitem{rev} P. R\'ev\'esz, Random
Walk in Random and Non-Random Environments,
2nd ed., World Scientific, Singapore, 2005.

\bibitem{rc} M. Rubinstein and R.H. Colby,
Polymer Physics, Oxford University Press, 2003.

\bibitem{rg} J. Rudnick and G. Gaspari, Elements
of the Random Walk, Cambridge University Press, 2004.

\bibitem{wat} H. Watanabe,
An asymptotic property of Gaussian processes. I.
{\em Trans. Amer. Math. Soc.} {\bf 148} (1970) 233--248.

\bibitem{zern} M.P.W. Zerner,
On the speed of a planar random walk avoiding its past convex hull,
{\em Ann. Inst. H. Poincar{\'e} Probab. Statist.} {\bf 41}
 (2005) 887--900.

\end{thebibliography}
\end{document}